\documentclass{amsart}          

\usepackage{amsmath,amsthm}     
\usepackage{amssymb,bbm}            
\usepackage{euscript}           
\usepackage{array,enumerate,calc,graphicx}
\usepackage[matrix,arrow,curve,frame]{xy}    
\usepackage[colorlinks]{hyperref}
\usepackage{hhline}

\xymatrixcolsep{1.9pc}                          
\xymatrixrowsep{1.9pc}
\newdir{ >}{{}*!/-5pt/\dir{>}}                  

\newtheorem{thm}[subsection]{Theorem}

\newtheorem{prop}[subsection]{Proposition}

\newtheorem{cor}[subsection]{Corollary}
\newtheorem{lemma}[subsection]{Lemma}

\theoremstyle{definition}  
\newtheorem{example}[subsection]{Example}

\newtheorem{remark}[subsection]{Remark}

  {\end{list}}

\newcommand{\dfn}{\textbf} 

\newcommand{\mdfn}[1]{\dfn{\mathversion{bold}#1}} 

\newcommand{\tens}              {\otimes}               

\newcommand{\iso}               {\cong}  

\newcommand{\cat}{\EuScript}    

\newcommand{\cS}{{\cat S}}

\newcommand{\field}[1]  {\mathbb #1} 

\newcommand{\F}         {\field F}

\newcommand{\N}         {\field N}
\newcommand{\Z}         {\field Z}

\renewcommand{\S}       {\field S}

\DeclareMathOperator*{\GL}{GL}

\DeclareMathOperator*{\im}{Im}

\DeclareMathOperator{\End}{End}

\DeclareMathOperator{\Sp}{Sp}
\DeclareMathOperator{\TO}{TO}
\DeclareMathOperator{\Eig}{Eig}

\DeclareMathOperator{\SYMP}{SYMP}
\DeclareMathOperator{\EVO}{EVO}
\DeclareMathOperator{\ODDO}{ODDO}

\newcommand{\ra}{\rightarrow}                   

\DeclareMathOperator{\Id}{Id}                            

\numberwithin{equation}{section}

\newenvironment{myequation}
  {\addtocounter{subsection}{1}\begin{eqnarray}}
  {\end{eqnarray}$\!\!$}

\DeclareMathOperator{\Iso}{Iso}

\DeclareMathOperator{\rank}{rank}

\newcommand{\tD}{\tilde{D}}
\newcommand{\talpha}{\tilde{\alpha}}

\newenvironment{bsmallmatrix}
  {\left [\begin{smallmatrix}}
  {\end{smallmatrix}\right ] }

\begin{document}

\title{Involutions in the topologists' orthogonal group}

\author{Daniel Dugger}
\address{Department of Mathematics\\ University of Oregon\\ Eugene, OR
97403} 

\email{ddugger@math.uoregon.edu}

\begin{abstract}
We classify conjugacy classes of involutions in the isometry groups of
nondegenerate, symmetric bilinear forms over the field $\F_2$.
The new component of this work 
focuses on the case of an orthogonal  form on an
even-dimensional space.  In this context
we show that the involutions satisfy a remarkable duality,
and we investigate several numerical invariants.
\end{abstract}

\maketitle

\tableofcontents

\section{Introduction}
Let $V$ be a finite-dimensional vector space over $\F_2$ equipped with a nondegenerate,
symmetric bilinear form $b$.  Write $\Iso(V)$ for the group of
isometries of $V$, meaning the group of
automorphisms 
$f\colon V\ra V$ such that $b(f(x),f(y))=b(x,y)$ for all
$x,y\in V$.  The goal of this paper is to classify the conjugacy classes of
involutions in $\Iso(V)$.  This involves three parts:
\begin{itemize}
\item[(P1)] Counting the number of conjugacy classes;
\item[(P2)] Giving a convenient set of representatives for the
conjugacy classes;
\item[(P3)] Giving a collection of computable invariants 
having the property that 
two involutions are in the same
conjugacy class if and only if they have the same invariants (for
brevity we will say that the invariants {\it completely separate\/} the conjugacy
classes).  
\end{itemize}

The motivation for solving this problem comes from an application in
topology.  Given a
compact manifold $M$ of even dimension $2d$, an involution $\sigma$ on $M$ induces
an involution $\sigma^*$ on $H^d(M;\F_2)$.  The cup product endows
this cohomology group with a nondegenerate, symmetric bilinear form,
and $\sigma^*$ is an isometry.  The conjugacy class of
$\sigma^*$ in $\Iso(H^d(M;\F_2))$ is an invariant of the topological
conjugacy class of $\sigma$.  That is, two involutions $\sigma$ and
$\theta$ on $M$ give isomorphic $\Z/2$-spaces only if $\sigma^*$ and
$\theta^*$ are conjugate inside of $\Iso(H^d(M;\F_2))$.  Thus, a
solution to (P3) yields topological invariants of the involution on
$M$.  These
invariants  play a  role in the classification of
$\Z/2$-actions on surfaces \cite{D}.

Beyond this initial motivation, however, this paper exists because
(P1)--(P3) have surprisingly nice answers.  The main work in the
present paper 
 occurs when $b$ is the standard dot product on $\F_2^n$, with $n$
 even. 
In this case the group $\Iso(V)$ turns out to have some remarkable
structures which aid in the classification.  In particular, the
involutions exhibit a surprising duality.

\medskip

To explain the results further, we first introduce two evident
invariants.  If $\sigma$ is in $\End(V)$ then we define
$D(\sigma)=\rank(\sigma+\Id)$ and call this the \mdfn{$D$-invariant}
of $\sigma$.  
It is an integral lift of the classical Dickson
invariant (see \cite[Theorem 11.43]{T}).  If $\sigma$ is an involution then the
Jordan form of $\sigma$ consists of $1\times 1$ blocks together with
$2\times 2$ blocks of the form $\begin{bsmallmatrix} 1& 1\\ 0 &
1\end{bsmallmatrix}$.  Then $D(\sigma)$ is simply the number of
$2\times 2$ blocks appearing, and of course this determines
the Jordan form.  So the $D$-invariant completely separates the conjugacy
classes for involutions in $\End(V)$.  Note that $0\leq D(\sigma)\leq
\frac{\dim V}{2}$ always (see Proposition~\ref{pr:D-half} below).  

For an involution $\sigma$ in $\Iso(V)$, we can consider the map
$V\mapsto \F_2$ given by $v\mapsto b(v,\sigma v)$.  It is non-obvious,
but easy to check, that this map is actually linear.  Let
$\alpha(\sigma)$ denote its rank, which is either $0$ or $1$.  This
is clearly an invariant of the conjugacy class of $\sigma$.

The pair $(V,b)$ is called \dfn{symplectic} if $b(x,x)=0$ for all
$x\in V$.  In this case $V$ has a symplectic basis, meaning a basis
$u_1,v_1,\ldots,u_n,v_n$ with $b(u_i,v_i)=1$ for all $i$, and all other pairings of
basis elements being zero.  It follows that $\Iso(V)$ is isomorphic to
the standard symplectic group $\Sp(2n)$ (all
matrix groups in this paper have matrix entries in $\F_2$, so we will
leave the field out of the notation).

In the symplectic case, the $D$ and $\alpha$ invariants completely
solve the conjugacy problem.  This is a classical result 
of Ashbacher-Seitz \cite{AS}, summarized in the following theorem.  
See also \cite[Section 6]{Dy1} for similar work.

\begin{thm}
\label{th:main-sp}
Suppose that $(V,b)$ is symplectic of dimension $2n$, and let
$\sigma\in \Iso(V)$ be an involution.
\begin{enumerate}[(a)]
\item If $\sigma$ and $\sigma'$ are two involutions in $\Iso(V)$, then
$\sigma$ and $\sigma'$ are conjugate if and only if
$D(\sigma)=D(\sigma')$ and $\alpha(\sigma)=\alpha(\sigma')$.
\item If $D(\sigma)$ is odd, then $\alpha(\sigma)=1$.  
\item The conjugacy classes of involutions in $\Iso(V)$ are in
bijective correspondence with the set of pairs $(D,\alpha)$ satisfying
$0\leq D\leq n$, $\alpha\in \{0,1\}$, and $\alpha=1$ when $D$ is odd.
The number of these conjugacy classes is
equal to
\[ \begin{cases}
\tfrac{3n+2}{2} & \text{if $n$ is even},\\
\tfrac{3n+1}{2} &\text{if $n$ is odd.}\\
\end{cases}
\]
\item Let $J=\begin{bsmallmatrix} 0 & 1 \\ 1 & 0\end{bsmallmatrix}$
and let $I=\begin{bsmallmatrix} 1 & 0 \\ 0 & 1\end{bsmallmatrix}$.
Let $M=\begin{bsmallmatrix} 1 & 0 & 1 & 1 \\
0 & 1 & 1 & 1 \\
1 & 1 & 1 & 0 \\
1 & 1 & 0 & 1
\end{bsmallmatrix}$.
Then the conjugacy classes of involutions in $\Sp(2n)$ are represented
by
\[ 
\begin{bmatrix} 
J \\
& I \\
&& I \\
&&& \ddots \\
&&&& I
\end{bmatrix},
\begin{bmatrix} 
J \\
& J \\
&& I \\
&&& \ddots \\
&&&& I
\end{bmatrix},\ldots,
\begin{bmatrix} 
J \\
& J \\
&& J \\
&&& \ddots \\
&&&& J
\end{bmatrix},
\]
together with
\[ I_{2n}, 
\begin{bmatrix} 
M \\
& I \\
&& \ddots \\
&&& I
\end{bmatrix},
\begin{bmatrix} 
M \\
& M \\
&& \ddots \\
&&& I
\end{bmatrix},
\begin{bmatrix} 
M \\
& M \\
&& \ddots \\
&&& M
\end{bmatrix}.
\]
The matrices in the first line have $\alpha=1$ and $D$ equal to
$1,2,3,\ldots,n$.  The matrices in the second line have
$\alpha=0$ and $D$ equal to $0,2,4,6,\ldots$ (terminating in either $n$
or $n-1$ depending on whether $n$ is even or odd).
\end{enumerate}
\end{thm}

If $V$ is an $n$-dimensional vector space over $\F_2$ and $b$ is a
nondegenerate symmetric bilinear form, then $(V,b)$ is either
symplectic or else it is isomorphic to $(\F_2^n,\cdot)$ where $\cdot$
is the standard dot product (see Proposition~\ref{pr:qspace-classify}).  
So it remains to discuss the latter
setting, which we call {\bf orthogonal\/}.  
It is tempting in this case to call $\Iso(\F_2^n,\cdot)$ an
{\it orthogonal group} and to denote it $O_n$, but this leads to some
trouble.  In characteristic $2$ situations, the theory of symmetric
bilinear forms and the theory of quadratic forms diverge.  Group
theorists, perhaps beginning with \cite{Di}, use ``orthogonal group''
to refer to the automorphisms of a quadratic form---this is different
from the group we need to study here.   So while the papers of Dye
\cite{Dy1,Dy2}, for example, classify conjugacy classes of involutions in orthogonal
groups over fields of characteristic 2, these are not the
orthogonal groups that are relevant to the problem we are trying to
solve.

Let us write $\TO(n)$ for $\Iso(\F_2^n,\cdot)$ and call it the {\bf
  topologists' orthogonal group} (for want of a better name).  It is
precisely the group of $n\times n$ matrices $A$ over $\F_2$ satisfying
$A^TA=I_n$.  Our goal is to describe conjugacy classes of involutions
in $\TO(n)$, for all $n$.  

When $n$ is odd, there is an isomorphism $\TO(n)\iso \Sp(n-1)$ (see
Proposition~\ref{pr:TO-odd}). 
So this case is again handled by Theorem~\ref{th:main-sp}.  
When $n$ is even, there is an isomorphism
\[ \TO(n)\iso M\rtimes \Sp(n-2)
\]
where $M$ is a certain modular representation of $\Sp(n-2)$ sitting
in a non-split short exact sequence
\[ 0 \ra \Z/2 \ra M \ra (\Z/2)^{n-2}\ra 0 \]
with the trivial representation on the left and the standard
representation on the right (see Corollary~\ref{co:sp-sd}).  It 
is this decomposition of $\TO(n)$ that
allows us to analyze the involutions, using the case of the symplectic
group as a starting point.  

When $n$ is even, the group $\TO(n)$ has a strange symmetry.  For $A\in
M_{n\times n}(\F_2)$, let $m(A)$ be the matrix obtained from $A$
by changing all the entries: $0$ changes to $1$, and $1$ changes to
$0$.  We call $m(A)$ the \dfn{mirror} of $A$.  Surprisingly, the
mirror of a matrix in $\TO(n)$ is again in $\TO(n)$.  
This map $m\colon\TO(n)\ra \TO(n)$ is of course not a group homomorphism (it
does not preserve the identity), but it does have the property that
\[ m(A)m(B)=AB \]
for all $A,B\in \TO(n)$.  In particular, if $A$ is an involution in
$\TO(n)$ then
$m(A)$ is also an involution in $\TO(n)$.  
One can also check that if $A$ and $B$ are conjugate inside of
$\TO(n)$ then so are $m(A)$ and $m(B)$.

Let us write $\tD(A)=D(mA)$ and
$\talpha(A)=\alpha(mA)$.  The invariants $D$, $\alpha$, $\tD$,
and $\talpha$ turn out to completely separate 
the  conjugacy classes of involutions.  For this reason, let us
define the \dfn{double-Dickson invariant}  of $A$ by 
\[ DD(A)=[D(A),\alpha(A),\tD(A),\talpha(A)]\in \Z^4.
\]
We will usually call this the \mdfn{$DD$-invariant}, for short.  
It turns out to always be the case that $|D(A)-\tD(A)|\leq 1$ (see
Proposition~\ref{pr:DD-props}).

There are numerous other invariants one can write down, and we give a thorough
discussion of these in Section~\ref{se:invariants}.  But the four invariants 
in $DD(A)$ seem to be the
most efficient way of capturing a complete set.

We next identify three families of involutions in $\TO(2n)$.  To this end,
if $A$ is an $n\times n$ matrix and $B$ is a $k\times k$ matrix let
us write $A\oplus B$ for the $(n+k)\times (n+k)$ block diagonal matrix
\[ \begin{bmatrix} A & O \\ O & B\end{bmatrix}.
\]
Let $I=\begin{bsmallmatrix} 1 & 0 \\ 0 & 1\end{bsmallmatrix}$ and
$J=\begin{bsmallmatrix} 0 & 1\\ 1 & 0 \end{bsmallmatrix}$.

The first family
consists of the matrices
\[ I^{\oplus(n-k)}\oplus J^{\oplus(k)}, \quad 1\leq k\leq n-1, \]
as well as their mirrors.  
The $DD$-invariants are given by
\[ DD\bigl (I^{\oplus(n-k)}\oplus J^{\oplus(k)} \bigr ) = [k,1,k+1,1],
\]
and for the mirrors one simply switches the first two coordinates with
the last two.  Note that there are $2(n-1)$ matrices in this family.   

The second family consists of the matrices 
\[
m\bigl (I^{\oplus(n-k)}\bigr ) \oplus J^{\oplus(k)}, \quad 0\leq k\leq n-1 \]
together with their mirrors.  Here the $DD$-invariants are given by
\[ DD\bigl (m(I^{\oplus(n-k)})\oplus J^{\oplus(k)} \bigr ) = 
\begin{cases} 
[k+1,0,k+1,1] & \text{if $k$ is odd,}\\
[k+1 ,0,k ,1] & \text{if $k$ is even.}
\end{cases}
\]
Note that there are $2n$ matrices in this family.  

Finally, our third family consists of the matrices
\[ m\bigl (I^{\oplus(n-k-1)}\oplus J^{\oplus(k) }\bigr )\oplus J, 
\qquad 1\leq k\leq n-2.
\]
These matrices turn out to be conjugate to their own mirrors, so we do
not include the mirrors this time.  The $DD$-invariants are
\[ DD
\Bigl (
m\bigl (I^{\oplus(n-k-1)}\oplus J^{\oplus(k) }\bigr )\oplus J
\Bigr ) = [k+2,1,k+2,1],
\]
and note that there are $n-2$ matrices in this family.  

Taking the three families together, we have produced $5n-4$ involutions.
One readily checks that all of their $DD$-invariants are different, so
this is a lower bound for the number of conjugacy classes of
involutions.  It turns out that there are no others:

\begin{thm}  Assume $n$ is even.  
\label{th:main-TO}
\begin{enumerate}[(a)]
\item The involutions $A$ and $B$ in $\TO(n)$ are in the same conjugacy
class if and only if $DD(A)=DD(B)$.
\item There are precisely $5n-4$ conjugacy classes of involutions in
$\TO(2n)$, and they are represented by the three families of matrices
\begin{align*}
& I^{\oplus(n-k)}\oplus J^{\oplus(k)}, \quad 1\leq k\leq n-1, \ \text{(together
  with their mirrors)}\\
& m(I^{\oplus(n-k)}) \oplus J^{\oplus(k)}, \quad 0\leq k\leq n-1 
\ \text{(together
  with their mirrors)}\\
& m\bigl (I^{\oplus(n-k-1)}\oplus J^{\oplus(k) }\bigr )\oplus J, \qquad 1\leq k\leq n-2.
\end{align*}
\end{enumerate}
\end{thm}

\medskip

\begin{example}
The group $\TO(6)$ has $23,\!040$ elements and 752 involutions,
falling into 11 
conjugacy classes.  The
possible $DD$-invariants are

\vspace{0.1in}

\begin{tabular}{cccccc}
$[1,1,2,1]$ & $[2,1,1,1]$ & $[2,1,3,1]$ & $[3,1,2,1]$ \\
$[1,0,0,1]$ & $[0,1,1,0]$  & $[2,0,2,1]$ & $[2,1,2,0]$
& $[3,0,2,1]$ & $[2,1,3,0]$ \\
$[3,1,3,1]$
\end{tabular}

\vspace{0.1in}

\noindent
where the three rows correspond to  the three families of involutions.
Here is a randomly chosen involution,
written side-by-side with its mirror:
\[
A=\begin{bmatrix}
0 & 0 & 1 & 1 & 1 & 0 \\
0 & 0 & 0 & 0 & 0 & 1 \\
1 & 0 & 1 & 1 & 0 & 0 \\
1 & 0 & 1 & 0 & 1 & 0 \\
1 & 0 & 0 & 1 & 1 & 0 \\
0 & 1 & 0 & 0 & 0 & 0
\end{bmatrix} \qquad
m(A)=\begin{bmatrix}
1 & 1 & 0 & 0 & 0 & 1 \\
1 & 1 & 1 & 1 & 1 & 0 \\
0 & 1 & 0 & 0 & 1 & 1 \\
0 & 1 & 0 & 1 & 0 & 1 \\
0 & 1 & 1 & 0 & 0 & 1 \\
1 & 0 & 1 & 1 & 1 & 1
\end{bmatrix}.
\]
The $\alpha$ and $\talpha$ invariants are easiest to read off: one
just looks along the diagonal.  The presence of $1$s along the
diagonal of $A$ implies $\alpha(A)=1$, and the presence of $0$s along the
diagonal of $A$ implies $\talpha(A)=1$ (since these $0$s lead to $1$s
along the
diagonal of $m(A)$).  Notice that this immediately puts $A$ in the
first or third family.  

Next, one readily computes $D(A)=\rank(A+\Id)=3$ and
$\tD(A)=D(mA)=\rank(mA+\Id)=3$.  So $DD(A)=[3,1,3,1]$, which
identifies the appropriate conjugacy class.  
\end{example}

\begin{remark}
One can naturally ask if the results of this paper extend to isometry
groups over other fields of characteristic two.  According to
\cite{AS}, this works fine in the symplectic
case---Theorem~\ref{th:main-sp} does not require that the ground field
be $\F_2$.  The same can therefore be said for the odd-dimensional
orthogonal case, as this case was secretly symplectic.  But for the
even-dimensional orthogonal case, the main methods in this paper
only work when the field is  $\F_2$.  In several places we use
constructions that make sense only
because certain maps that
satisfy $F(\lambda v)=\lambda^2 F(v)$ actually turn out to be linear; 
this  cannot possibly happen over other fields.    
\end{remark}

\subsection{Organization of the paper}
In Section~\ref{se:two} we develop the basics of nondegenerate 
symmetric bilinear forms over $\F_2$ and their isometries.  We define
the mirror operation, and we explore the connections between the
groups $\TO(k)$ and $\Sp(n)$.  Section~\ref{se:invariants}
introduces a slew of invariants for involutions, and establishes their
basic properties.  In Section~\ref{se:main} we prove Theorem~
\ref{th:main-TO}.  Finally, Section~\ref{se:dsum} provides formulas
for how the $DD$-invariant behaves under direct sums; these are very
useful in applications.  Unfortunately  this is the
most tedious part of the paper, as the formulas involve many cases and
are not very enlightening.

\subsection{Acknowledgments}
The author is grateful to Bill Kantor for some extremely helpful
correspondence.


\section{Background}
\label{se:two}

Let $F$ be a field.  By a \dfn{bilinear space} over $F$ we mean a
finite-dimensional 
vector space $V$ together with a nondegenerate symmetric bilinear form $b$ on $V$.
Recall that nondegenerate means  no nonzero vector is orthogonal to every
vector in $V$.  Bilinear spaces are more commonly called {\it quadratic
spaces\/} in the literature, but since the theories of quadratic forms
and bilinear forms diverge in characteristic two the 
terminology chosen here leads to less confusion.  

If $a\in F$ we write $\langle a\rangle$ for the one-dimensional vector
space $F$ equipped with the bilinear form $b(x,y)=axy$.  
We write $H$
for $F^2$, with standard basis $\{e_1,e_2\}$, equipped with the
bilinear form where $b(e_1,e_1)=b(e_2,e_2)=0$ and $b(e_1,e_2)=1$.
Write
$n\langle 1\rangle$ for $\langle 1\rangle\oplus \langle 1\rangle
\oplus \cdots \oplus \langle 1\rangle$ and  $nH$ for $H\oplus H\cdots
\oplus H$ ($n$ summands in each case).  

A  bilinear space $(V,b)$ is called \dfn{symplectic} if $b(v,v)=0$ for all $v$
in $V$.  Any symplectic space is isomorphic to $nH$ for some $n$, by
\cite[Corollary 3.5]{HM}.  The proof is simple: choose any nonzero
$x\in V$, and then choose a $y\in V$ such that $b(x,y)=1$.  Take the
orthogonal complement of $\F_2\langle a,b\rangle$ in $V$ and continue
by induction.  

\begin{prop}
\label{pr:qspace-classify}
Every nondegenerate bilinear space over $\F_2$ is
isomorphic to either $nH$ or
$n\langle 1\rangle$, for some $n\geq 1$.  
\end{prop}

\begin{proof}
Let $(V,b)$ be a nondegenerate bilinear space over $\F_2$.  If $V$ is
symplectic then we are done, so
we may assume that $V$ contains a vector $x_1$ such that
$b(x_1,x_1)=1$.  Take the orthogonal complement of $x_1$ and continue
inductively, until one obtains a space that is symplectic.  
This shows that $V$ is isomorphic to $k\langle 1\rangle
\oplus rH$, for some $k$ and $r$.

We will be done if we can show that $\langle 1\rangle\oplus H\iso
3\langle 1\rangle$, since then if $k\neq 0$ any copy of $H$ in the
decomposition of $V$ can be replaced with $2\langle 1\rangle$.   Suppose that
$x,y,z$ is a basis for a space such that 
\[ b(x,x)=1=b(y,z), \quad b(x,y)=b(x,z)=b(y,y)=b(z,z)=0.
\]
It is easy to check that $x+y$, $x+z$, and $x+y+z$ is an orthonormal
basis for the same space.  
\end{proof}

\begin{remark}
Note that $n\langle 1\rangle$ is simply $\F_2^n$ with the standard dot
product form.  We will usually denote this $(\F_2^n,\cdot)$, and will
write $e_1,\ldots,e_n$ for the standard orthonormal basis.  
\end{remark}

From now on we only work over the field $\F_2$.  
If a bilinear space $(V,b)$ is isomorphic to $nH$ we will
write $\Sp(V)=\Iso(V,b)$.  If $(V,b)$ is isomorphic to
$(\F_2^n,\cdot)$ we say that $V$ is \dfn{orthogonal} and write
$\TO(V)=\Iso(V,b)$.  
We will also use the notation
$\Sp(2n)$ for the group of isometries of $nH$, and  $\TO(n)$ for
the group of isometries of $n\langle 1\rangle$.  Note that we may
identify $\Sp(2n)$ with the usual group of $2n\times 2n$ symplectic
matrices over $\F_2$, and we may identity $\TO(n)$ with the group 
of $n\times n$ matrices $A$ over $\F_2$ such that $AA^T=I_n$.

If $(V,b)$ is a bilinear space over $\F_2$, then $v\mapsto b(v,v)$ gives a
linear map $f\colon V\ra \F_2$.  Note that this depends on the fact that
$\lambda^2=\lambda$ for all $\lambda\in \F_2$.  Since $b$ is
nondegenerate, the adjoint of $b\colon V\tens V\ra \F_2$ is an
isomorphism $V\ra V^*$.  Taking the preimage of $f$ under this
isomorphism, 
we find that
 there is a unique vector
$\Omega\in V$ with the property that
\[ b(\Omega,v)=b(v,v) \quad\text{for all $v\in V$}.
\]
We call $\Omega$ the \dfn{distinguished vector} in $V$.  
Note that when $(V,b)=(\F_2^n,\cdot)$, the distinguished vector is
$[1,1,\ldots,1]$.  The bilinear space $(V,b)$ is symplectic if and
only if $\Omega=0$.

Observe that every isometry of $(V,b)$ must necessarily fix
$\Omega$, and therefore maps $\langle \Omega\rangle^\perp$ into
itself.

\begin{prop}
\label{pr:TO-odd}
When $n$ is odd one has $\TO(n)\iso \Sp(n-1)$.  
\end{prop}

\begin{proof}
Let $U=\langle \Omega \rangle^{\perp}$.  When $n$ is
odd we have a decomposition $\F_2^n=U\oplus \langle \Omega\rangle$.  Every
element of $\TO(n)$ fixes $\Omega$ and therefore maps $U$ to $U$, so we
have $\TO(n)\iso \Iso(U)$.  But the
space $U$ is symplectic, so $\Iso(U)\iso \Sp(n-1)$.  
\end{proof}

It is easy to count the number of elements in $\TO(n)$.  The following
result is classical, but a nice reference is \cite{M}:

\begin{prop}
\label{pr:orders}
For $n\geq 1$ one has
\[ |\TO(2n)|=2^n\cdot(4^n-4^1)(4^n-4^2)\cdots(4^n-4^{n-1}),
\]
\[ |\Sp(2n)|=|\TO(2n+1)|=2^n\cdot (4^n-4^0)(4^n-4^1)(4^n-4^2)\cdots(4^n-4^{n-1}).
\]
\end{prop}

\subsection{Mirrors}

Let $(V,b)$ be an orthogonal space and assume that $\dim V$ is even.
This condition forces $b(\Omega,\Omega)=0$.  
If $L\colon V\ra V$ is an isometry, define $mL\colon V\ra V$ by
\[ mL(v)=L(v)+b(v,\Omega)\Omega.
\]
We call $mL$ the \dfn{mirror} of $L$.  Clearly $mL$ is still
linear, and it is also still an isometry:
\begin{align*}
b(mL(v),mL(w))&=b(Lv+b(v,\Omega)\Omega,Lw+b(w,\Omega)\Omega)\\
& =b(Lv,Lw)+b(Lv,\Omega)b(w,\Omega)+b(v,\Omega)b(Lw,\Omega)\\
&= b(v,w)+b(v,\Omega)b(w,\Omega)+b(v,\Omega)b(w,\Omega) \\
&= b(v,w).
\end{align*}
In the third equality we used that $L(\Omega)=\Omega$ and so
$b(Lv,\Omega)=b(Lv,L\Omega)=b(v,\Omega)$.  

Of course $m\colon \Iso(V)\ra \Iso(V)$ is not a group homomorphism;
for example, it does not preserve the identity.  But it satisfies the
following curious property:

\begin{prop}
\label{pr:mirror-invo}
If $F,L\in \Iso(V,b)$ then $(mF)(mL)=FL$.  In particular, if $F$ is an
involution then $mF$ is also an involution.  
\end{prop}

\begin{proof}
For $v\in V$ we compute that
\begin{align*}
mF(mL(v))=mF\bigl (Lv+b(v,\Omega)\Omega\bigr )&=F\bigl (Lv+b(v,\Omega)\Omega\bigr 
)+b(Lv,\Omega)\Omega\\
&= FL(v)+b(v,\Omega)F(\Omega)+b(Lv,\Omega)\Omega.
\end{align*}
Now use the facts that $F(\Omega)=\Omega$ and 
$b(Lv,\Omega)=b(Lv,L\Omega)=b(v,\Omega)$.  
\end{proof}

\begin{remark}
For $(\F^{2n}_2,\cdot)$, recall that $\Omega=[1,1,\ldots,1]$.  So
$m\colon \TO(2n)\ra \TO(2n)$ is the function that adds $\Omega$ to
each column of a matrix $A\in \TO(2n)$.  Clearly this amounts to
changing every entry in $A$, from a $0$ to a $1$ or from a $1$ to a $0$.
\end{remark}

The following result shows that if two isometries are conjugate, then
their mirrors are also conjugate:

\begin{prop}
\label{pr:mirror-conjugacy}
Suppose $A,P\in \Iso(V,b)$.  Then $m(P A P^{-1})=P\circ 
m(A) \circ P^{-1}$.  
\end{prop}

\begin{proof}
The isometry $m(PAP^{-1})$ is given by
\[ v\mapsto PAP^{-1}(v) + b(v,\Omega)\Omega.
\]
The isometry $P\circ m(A)\circ P^{-1}$ is given by
\[ v \mapsto P\bigl ( AP^{-1}v + b(P^{-1}v,\Omega)\Omega\bigr )=PAP^{-1}v +
b(P^{-1}v,\Omega)P(\Omega).
\]
Now use that $P(\Omega)=\Omega$  and
$b(P^{-1}v,\Omega)=b(v,P\Omega)=b(v,\Omega)$.  
\end{proof}

\subsection{More on the symplectic group}
We first state a simple result that will be needed later:

\begin{prop}
\label{pr:sp-trans}
Let $(V,b)$ be a symplectic bilinear space over $\F_2$.  Then
$\Iso(V,b)$ acts transitively on $V-\{0\}$.
\end{prop}

\begin{proof}
This is surely standard.  See \cite[Lemma 4.14]{D} as one source
for a proof.
\end{proof}

If $F$ is a field and $V$ is a vector space,
recall that a quadratic form on $V$ is a function $q\colon V\ra F$ such that 
$q(\lambda v)=\lambda^2 q(v)$ and $q(v+w)-q(v)-q(w)$ is bilinear.  
When $F=\F_2$ one has $\lambda^2=\lambda$ for all scalars, so the
first condition simplifies.

If $(V,b)$ is a symplectic bilinear space then a \dfn{semi-norm} 
is a
quadratic form $q\colon
V\ra F$ such that $q(v+w)=q(v)+q(w)+b(v,w)$ for all $v,w\in V$.  Such
a $q$ cannot be unique: adding any linear form to $q$
gives another semi-norm.  In fact the set of all semi-norms for $b$ is a
torsor for the group $V^*$ of linear forms on $V$.  The only
nontrivial statement in all of this is the assertion that 
a semi-norm exists at
all.  To see this, consider $nH$ with the standard symplectic basis
$\{f_i,g_i\}$.  
Define $q(\sum x_if_i+y_ig_i)=\sum_i x_iy_i$.  One readily checks that
this is a semi-norm.  

Fix a semi-norm $q$ for $(V,b)$, and let $A\in \Sp(V)$.  The function
$V\ra \F_2$ given by $v\mapsto q(v)+q(Av)$ is readily checked to be
linear.  Write $\S_q A$ for this linear functional.  Note that
$\S_q(\Id)$ is zero.  If $q'$ is another semi-norm for $(V,b)$ then
$\S_{q'}A=\S_qA + (q+q')$.

\begin{prop}
\label{pr:Sq}
For any $A,B\in \Sp(V)$ one has $\S_q(AB)=\S_q(B)+\S_q(A)\circ B$.  
\end{prop}

\begin{proof}
One simply computes that
\begin{align*}
\S_q(AB)(v)=q(v)+q(ABv)& =q(v)+q(Bv)+q(Bv)+q(ABv)\\
&=(\S_qB)(v)+(\S_qA)(Bv).
\end{align*}
\end{proof}

Let $M_V=V\oplus \F_2$.  Define an action of $\Sp(V)$ on $M_V$ by
\begin{myequation}
\label{eq:action}
 A\cdot (v,\lambda)=\bigl (A(v),(\S_q A)(v)+\lambda\bigr ).  
\end{myequation}
We leave the reader to check that this is indeed a group action, using
Proposition~\ref{pr:Sq}.  Clearly $M_V$ sits in a short exact sequence
$0\ra \F_2 \ra M_V \ra V\ra 0$
where $\F_2$ has the trivial action of $\Sp(V)$ and $V$ has the
standard action.  

If $q'$ is another semi-norm for $(V,b)$ then we get two actions on
$M_V$; let us call them $M_V(q)$ and $M_V(q')$.  These are isomorphic
$\Sp(V)$-spaces, via the isomorphism $M_V(q)\ra M_V(q')$ given by
$(v,\lambda)\mapsto (v,q(v)+q'(v)+\lambda)$.
Recall that $q+q'$ is linear.

\begin{example}
\label{ex:Sp-rep}
It is useful to understand how these constructions look in the
concrete world of matrices.  If $V=nH$ then one possible semi-norm
is $q([x_1,y_1,\ldots,x_n,y_n]=x_1y_1+\cdots+x_ny_n$.  The
representation $M_V$ is a group homomorphism $\Sp(2n)\ra \GL(2n+1)$.
For $n=2$
this is 

\begingroup
\renewcommand*{\arraystretch}{0.9}

\[
\begin{bmatrix} a & b & c & d \\
e & f & g & h \\
i & j & k & l \\
m & n & p & q
\end{bmatrix} \mapsto
\begin{bmatrix} a & b & c & d & 0\\
e & f & g & h &0  \\
i & j & k & l &0 \\
m & n & p & q &0\\
ae+im & bf+jn & cg+kp & dh+lq & 1
\end{bmatrix} 
\]
\endgroup

\noindent
and the pattern for larger $n$ is the evident one.  An industrious
reader can check by hand that this is indeed a group homomorphism, but
it is not obvious from the above formula!
\end{example}

Continue to assume that $(V,b)$ is symplectic, and
now consider the bilinear space $\hat{V}=V\oplus \langle 1\rangle \oplus
\langle 1\rangle$.  We will still write $b$ for the bilinear form on
this larger space.  Let $e$ and $f$ be the two basis elements
corresponding to the two $\langle 1\rangle$ summands, so that
$b(e,e)=b(f,f)=1$, $e,f\in V^{\perp}$, and $b(e,f)=0$.    Note that
$\hat{V}$ is an orthogonal space by
Proposition~\ref{pr:qspace-classify} (as it is certainly not symplectic), and one
readily checks that $e+f$ is the distinguished vector $\Omega$.  
For this reason it will 
be a
little more convenient for us to use the basis $\{\Omega,f\}$ instead of
$\{e,f\}$.  
Note that
$b(\Omega,\Omega)=0$ and $b(\Omega,f)=b(f,f)=1$.  

There is an evident homomorphism $j\colon \Sp(V)\ra \Iso(\hat{V})$.
If $A\in \Sp(V)$ then $j(A)\colon \hat{V}\ra \hat{V}$ fixes $\Omega$ and
$f$, and acts as $A$ on the $V$ summand.  

For $(v,\lambda)\in M_V$ define $\phi_{(v,\lambda)}\colon \hat{V}\ra
\hat{V}$ by
\begin{align*}
&\phi_{(v,\lambda)}(w)=w+b(w,v)\Omega\  \text{for $w$ in $V$}, \\
&\phi_{(v,\lambda)}(\Omega)=\Omega\\
&\phi_{(v,\lambda)}(f)=v+(\lambda+q(v))\Omega+f.
\end{align*} 
It is easy to check that $\phi_{(v,\lambda)}$ is an isometry, and that
$\phi$ gives a group homomorphism
$\phi\colon M_V \ra \Iso(\hat{V})$.
Moreover, if $A\in \Sp(V)$ and $x\in M_V$
then
\[ \phi(A\cdot x)=j(A)\phi(x)j(A)^{-1}.
\]
This verifies that we get a group map $\theta\colon M_V\rtimes
\Sp(V)\ra \Iso(\hat{V})$ by defining $\theta(x,A)=\phi(x)j(A)$.

\begin{prop}
The map $\theta\colon M_V\rtimes \Sp(V)\ra \Iso(\hat{V})$ is an
isomorphism.  
\end{prop}

\begin{proof}
Let $\dim V=2n$.  
Using Proposition~\ref{pr:qspace-classify}, 
the bilinear space $\hat{V}$ is isomorphic to
$(\F_2^{2n+2},\cdot)$.  So $\Iso(\hat{V})$ is isomorphic to
$\TO(2n+2)$.   One then readily checks using
Proposition~\ref{pr:orders} that the domain and target of $\theta$
have the same order.  So it suffices to show that $\theta$ is
injective.  

Let $(v,\lambda)\in M_V$ and $A\in \Sp(V)$, and assume that
$\phi(v,\lambda)j(A)=\Id$.  The transformation $j(A)$ fixes $f$, and
$\phi(v,\lambda)$ sends $f$ to $v + (\lambda+q(v))\Omega + f$.  It follows
that $v=0$ and $\lambda+q(v)=0$, which in  turn implies $\lambda=0$.
Therefore $\phi(v,\lambda)=\Id$ and so $j(A)=\Id$, which means $A=\Id$.  
\end{proof}

\begin{cor}
\label{co:sp-sd}
There is a group isomorphism $\TO(2n)\iso M \rtimes \Sp(2n-2)$, where
$M$ is the representation of $\Sp(2n-2)$ on $\F_2^{2n-1}$ described in
Example~\ref{ex:Sp-rep}.
\end{cor}

Recall the mirror operation $m\colon \Iso(\hat{V})\ra \Iso(\hat{V})$.
In view of the isomorphism $\theta$, there should be a corresponding
operation on $M_V\rtimes \Sp(V)$.  To construct this,
define a set map $\tilde{m}\colon M_V\ra M_V$ by
$\tilde{m}(v,\lambda)=(v,\lambda+1)$.  Extend this to a set map
$m\colon M_V\rtimes \Sp(V)\ra M_V\rtimes \Sp(V)$ by
\[ m(x,A)=(\tilde{m}(x),A).
\]

\begin{prop} The diagram
\[ \xymatrix{
M_V\rtimes \Sp(V) \ar[r]^-\theta\ar[d]_m & \Iso(\hat{V}) \ar[d]^m \\
M_V\rtimes \Sp(V) \ar[r]^-\theta & \Iso(\hat{V})
}
\]
is commutative.
\end{prop}

\begin{proof}
Pick $(v,\lambda)\in M_V$ and $A\in \Sp(V)$.  Then
$m\theta((v,\lambda),A)$ has the following behavior:
\[ \begin{cases}
w\mapsto Aw+b(Aw,v)\Omega \ \  \text{if $w\in V$},\\
\Omega\mapsto \Omega \\
f\mapsto v+(\lambda+q(v)+1)\Omega + f + \Omega.
\end{cases}
\]
By inspection this is the same behavior as $\theta((v,\lambda+1),A)$.  
\end{proof}


\section{Invariants}
\label{se:invariants}

Let $(V,b)$ be a bilinear space over $\F_2$.  In this section we study
various numerical invariants that can be assigned to involutions in
$\Iso(V,b)$, having the property that they are constant on
conjugacy classes.  Our focus is mainly on the case where $(V,b)$ is
orthogonal, but it is convenient to  discuss the symplectic case at the
same time.  

\medskip

Let $\sigma\in \Iso(V,b)$.  Define the \dfn{Dickson invariant} 
\mdfn{$D(\sigma)$} to be the rank of
$\sigma+\Id$.  Note that this is clearly invariant under  conjugacy
in $\GL(V)$, and therefore also under conjugacy in the smaller group
$\Iso(V,b)$.

\begin{prop}
\label{pr:D-half}
For any involution in $\Iso(V,b)$ one has $0\leq D(\sigma)\leq
\frac{\dim V}{2}$.  
\end{prop}

\begin{proof}
Observe that $(\Id+\sigma)^2=0$, or equivalently 
$\im(\Id+\sigma)\subseteq \ker(\Id+\sigma)$.  So
\[ D(\sigma)=\dim \im(\Id+\sigma)\leq \dim \ker(\Id+\sigma)=\dim
V-D(\sigma).
\]
\end{proof}

The map $F_\sigma\colon V\ra \F_2$ given by $v\mapsto b(v,\sigma v)$
is linear, 
since
\[ F_\sigma(v+w)=b(v+w,\sigma v+\sigma w)=b(v,\sigma v)+b(w,\sigma
w)+b(w,\sigma v)+b(v,\sigma w)
\]
and $b(w,\sigma v)=b(\sigma w,\sigma^2 v)=b(\sigma w,v)$.  Note that
here we have used both that $\sigma$ is an involution and an isometry.
Define the \mdfn{$\alpha$-invariant} \mdfn{$\alpha(\sigma)$} to be the
rank of $F_\sigma$.  Deconstructing this, we have
\[ \alpha(\sigma)=
\begin{cases}
1 & \text{if there exists a $v\in V$ such that $b(v,\sigma v)=1$},\\
0 & \text{otherwise.}
\end{cases}
\]
It is easy to check that the $\alpha$-invariant is constant on
conjugacy classes in $\Iso(V,b)$.  

As we saw in Theorem~\ref{th:main-sp}, it is a classical result that the
pair $(D(\sigma),\alpha(\sigma))$ completely separates conjugacy classes
when $(V,b)$ is symplectic.  So let us now focus on the case where
$(V,b)$ is orthogonal.  Here we may use the mirror operation $m\colon
\Iso(V,b)\ra \Iso(V,b)$, which we know sends involutions to
involutions (Proposition~\ref{pr:mirror-invo}) 
and preserves the conjugacy relation 
(Proposition~\ref{pr:mirror-conjugacy}).  If $\sigma\in
\Iso(V,b)$ is an involution define
\[ \tD(\sigma)=D(m\sigma),\qquad \talpha(\sigma)=\alpha(m\sigma).
\]
Moreover, define the \dfn{double Dickson invariant} (or
\mdfn{$DD$-invariant}, for short) to be the
$4$-tuple
\[
DD(\sigma)=[D(\sigma),\alpha(\sigma),\tD(\sigma),\talpha(\sigma)]\in
\N \times \Z/2\times \N\times \Z/2.
\]
This $4$-tuple is constant on conjugacy classes.  

Recall that $(m \sigma)(v)=\sigma(v)+b(v,v)\Omega$.  Then 
$b(v,(m\sigma)(v))=b(v,\sigma(v))+b(v,v)$.
Therefore 
\[ \talpha(\sigma)=
\begin{cases}
1 & \text{if there exists $v\in V$ such that
  $b(v,\sigma(v))=b(v,v)+1$,} \\
0 & \text{otherwise.}
\end{cases}
\]

\begin{remark}
\label{re:alpha}
Suppose that $e_1,\ldots,e_n$ is an orthonormal basis for $V$, and
that $A$ is the matrix of $\sigma$ with respect to this basis.  Then
one readily checks that
\[ \alpha(\sigma)=\begin{cases}
1 & \text{if $A$ has at least one $1$ on its diagonal,}\\
0 & \text{otherwise},
\end{cases}
\]
and
\[ \talpha(\sigma)=\begin{cases}
1 & \text{if $A$ has at least one $0$ on its diagonal,}\\
0 & \text{otherwise}.
\end{cases}
\]
\end{remark}

\begin{prop}
\label{pr:DD-props}
Let $\sigma$ be an involution in $\Iso(V,b)$, where $(V,b)$ is
orthogonal.  Then $|D(\sigma)-\tD(\sigma)|\leq 1$, and 
one cannot have $\alpha(\sigma)=\talpha(\sigma)=0$.
\end{prop}

\begin{proof}
The second statement follows immediately from Remark~\ref{re:alpha}.
For the first, let $e_1,\ldots,e_n$ be an orthonormal basis for $V$,
so that $\Omega=\sum_i e_i$.  Let $A$ be the  matrix for $\sigma$ with
respect to this basis, and let $u_1,\ldots,u_n$ denote the columns of
$A+\Id$.  Then the columns of $m(A)+\Id$ are the vectors $u_i+\Omega$.
Let $U=\F_2\langle u_1,\ldots,u_n\rangle  \subseteq V$, and
$W=\F_2\langle u_1+\Omega,\ldots,u_n+\Omega\rangle \subseteq V$.  Then
$D(\sigma)=\dim U$ and $\tD(\sigma)=\dim W$.  

But $U=\F_2\langle u_1,u_1-u_2,u_1-u_3,\ldots, u_1-u_n\rangle$ and
$W=\F_2\langle u_1+\Omega,u_1-u_2,\ldots,u_1-u_n\rangle$.  It is now clear
that $|\dim U-\dim W|\leq 1$.
\end{proof}

\subsection{Other invariants}
The $DD$-invariant is the main construct that will be used in the rest
of the paper.  However, one can easily write down a multitude of other
invariants for conjugacy classed of involutions.  Our next goal will
be to give a 
thorough exploration of these.  We should say upfront, though,
that the results of this section are not needed for the main
classification result. 
Nevertheless, they merit inclusion here because they shed some
light on the broader story surrounding the $DD$-invariant.  In
addition, they are useful for calculating how the $DD$-invariant
behaves under direct sums (see Theorem~\ref{th:DD-osum} below).

We begin with a naive example.  Given an involution $\sigma$ in
$\Iso(V,b)$, consider the set
\[ S_\sigma=\{v\in V\,|\, b(v,\sigma v)=0\} \subseteq V.
\]
Since $V$ is finite-dimensional over $\F_2$, $S_\sigma$ is finite.  The order
$|S_\sigma|$ is clearly an invariant of $\sigma$: if $f\colon (V,b)\ra
(W,b')$ is an isomorphism of bilinear spaces and $\sigma'$ is an
involution in $\Iso(W,b')$ such that $f\sigma=\sigma' f$, then clearly
$f$ maps $S_\sigma$ bijectively onto $S_{\sigma'}$.  In particular,
applying this when $(W,b')=(V,b)$ shows that $|S_\sigma|$ is an
invariant of the conjugacy class of $\sigma$ in $\Iso(V,b)$.  

At this point it is clear how to generalize.  Any property $P$ of vectors
$v\in V$ that can be expressed entirely in terms of $b$ and $\sigma$
leads to a  set $S_\sigma(P)$ and a conjugacy invariant
$|S_\sigma(P)|$.  One can easily write down three basic instances of
such a $P$, and in the case that $(V,b)$ is orthogonal 
there is one more that is slightly less-evident:
\[ b(v,v)=0, \qquad b(v, \sigma v)=0,\qquad v=\sigma(v), \qquad
v=\sigma(v)+\Omega.  
\]
By taking combinations of these four properties and their negations,
one can make $2^4=16$ different invariants---but only eight of these
turn out to be interesting, as some of the combinations are either  mutually 
inconsistent or duplicate other combinations.  
Restricting now only to the orthogonal case, the following table
introduces eight invariants and shows their values on the 16 conjugacy
classes of involutions in $\TO(8)$ (these numbers were generated by
computer).
For typographical reasons we write $b(x,y)$ as $x\cdot y$ in this table.

\vspace{0.15in}

\begin{tabular}{c||c|c|c||c|c|c||c|c}
$DD$ & $I_1$ & $I_2$ & $I_3$ & $I_4$ & $I_5$ & $I_6$ & $I_7$ & $I_8$ \\
\hline 
 & \parbox{0.3in}{$\scriptstyle{v\cdot v=0}$\\ $\scriptstyle{v=\sigma v}$ \\} 
& \parbox{0.3in}{$\scriptstyle{v\cdot v=0}$\\ $\scriptstyle{v\cdot \sigma v=0}$ \\ $\scriptstyle{v\neq \sigma
  v}$ \\} 
& \parbox{0.3in}{$\scriptstyle{v\cdot v=0}$\\ $\scriptstyle{v\cdot
    \sigma v=1}$ \\} 
& \parbox{0.4in}{$\scriptstyle{v\cdot v=1}$\\ $\scriptstyle{v\cdot
    \sigma v=0}$ \\} 
& \parbox{0.4in}{$\scriptstyle{v\cdot v=1}$\\ $\scriptstyle{v= \sigma
    v}$ \\} 
& \parbox{0.4in}{$\scriptstyle{v\cdot v=1}$\\ $\scriptstyle{v\cdot \sigma v=1}$ \\ $\scriptstyle{v\neq \sigma
  v}$ \\} 
& \parbox{0.4in}{$\scriptstyle{v=\sigma v+\Omega}$\\
  $\scriptstyle{v\cdot v=0}$ \\} 
& \parbox{0.4in}{$\scriptstyle{v=\sigma v+\Omega}$\\
  $\scriptstyle{v\cdot v=1}$ \\} 
\\ 
\hline
0110 & 128 & 0 & 0 & 0 & 128 & 0 & 0 & 0\\
1001 & 128 & 0 & 0 & 128 & 0 & 0 & 0 & 128 \\
1121 & 64  & 0 & 64 & 64 & 64 & 0 & 0 & 0\\
2111 & 64 & 0 & 64 & 64 & 0 &64 & 0 & 64\\
2021 & 64 & 64 & 0 & 128 & 0 & 0 & 64 & 0\\
2120 & 64 & 64 & 0 & 0 & 0 & 128 & 64 & 0\\
2130 & 32 & 96 & 0 & 0 & 32 & 96 & 0 & 0\\
3021 & 32 & 96 & 0 & 128 & 0 & 0 & 0 & 32\\
2131 & 32 & 32 & 64 & 64 & 32 & 32 & 0 & 0 \\
3121 & 32 & 32 & 64 & 64 & 0 & 64 & 0 & 32\\
3131 & 32 & 32 & 64 & 64 & 0 & 64 & 32 & 0\\
3141 & 16 & 48 & 64 & 64 & 16 & 48 & 0 & 0\\
4131 & 16 & 48 & 64 & 64 & 0 & 64 & 0 & 16\\
4041 & 16 & 112 & 0 & 128 & 0 & 0 & 16 & 0\\
4140 & 16 & 112 & 0 & 0 & 0 & 128 & 16 & 0\\
4141 & 16 & 48 & 64 & 64 & 0 & 64 & 16 & 0\\
\end{tabular}

\vspace{0.15in}

Certain properties of these invariants are immediately evident---for
example, the numbers are always even and most of them are powers of
$2$.  
To explain these, note that the
functions $L_1(v)=v\cdot v$, $L_2(v)=v\cdot \sigma(v)$, and
$L_3(v)=v+\sigma(v)$ are all linear.  So the solution spaces to
$L_i(v)=0$ are linear subspaces of $V$, and the solution spaces to
$L_i(v)\neq 0$ are affine subspaces of $V$ when $i=1,2$ and the
complement of a subspace when $i=3$.  Likewise, the solution space to
$L_3(v)=\Omega$ is an affine space.  This clearly implies that all the
invariants are even.  Even more, it shows that except for $I_2$ and
$I_6$ the invariants always yield powers of $2$.  (As we shall see
shortly, $I_2$ and $I_6$ should really be left out of the story
altogether as they can be obtained as
 linear combinations of the other invariants).  

We can push the above idea a little further.  Let $S_1=\{v\in V\,|\,
v\cdot v=0, v\cdot \sigma(v)=0\}$.  This is a linear subspace of $V$,
and $I_1=|S_1|$.
Analogously, let $S_j\subseteq V$ be the subset that defines the
invariant $I_j$.
One readily checks that vector addition gives an action of the group 
$(S_1,+)$ on $S_j$: that is, if $v\in S_1$ and $w\in S_j$ then
$v+w\in S_j$.  Moreover, it is clearly a free action.  This shows that
$I_j$ is always a multiple of $I_1$.

The following proposition summarizes various relations amongst the
$I_j$ invariants and the $DD$-invariant:

\begin{prop} 
\label{pr:I-invs}
Let $(V,b)$ be an orthogonal bilinear space of even dimension, and let
$\sigma$ be an involution in $\Iso(V,b)$.  Write $I_1=I_1(\sigma)$,
etc.  Then 
\begin{enumerate}[(a)]
\item $I_1+I_2+I_3=2^{\dim V-1}=I_4+I_5+I_6$.  
\item $I_5,I_7,I_8\in \{0,I_1\}$.
\item At most one of $I_5$, $I_7$, and $I_8$ is nonzero.
\item $D(\sigma)=\dim V-\log_2(I_1+I_5)$.
\item $\tD(\sigma)=\dim V-\log_2(I_1+I_8)$.
\item $\alpha(\sigma)=0 \iff (\text{$I_4=2^{\dim V-1}$ and
  $I_3=0$})\iff I_4=2^{\dim V-1}$. 
\item $\talpha(\sigma)=0\iff I_3=I_4=0 \iff I_4=0$.  
\item $\log_2(I_1)=\dim V-\max\{D(\sigma),\tD(\sigma)\}$.  
\item $I_3=\begin{cases}
0 & \text{if $\alpha(\sigma)\neq \talpha(\sigma),$}\\
2^{\dim V-2} & \text{if $\alpha(\sigma)=\talpha(\sigma)$}
\end{cases}$ \quad and\quad
 $I_4=\begin{cases}
0 & \text{if $\alpha(\sigma)>\talpha(\sigma)$}, \\
2^{\dim V-1} & \text{if $\alpha(\sigma)<\talpha(\sigma)$}, \\
2^{\dim V-2} & \text{if $\alpha(\sigma)=\talpha(\sigma)$}.  
\end{cases}$
\item $I_5=\begin{cases}
0 & \text{if $D(\sigma)\geq \tD(\sigma)$}, \\
2^{\dim V-1-D(\sigma)} & \text{if $D(\sigma)<\tD(\sigma)$.}
\end{cases}$
\item
$I_7=
\begin{cases}
0 & \text{if $D(\sigma)\neq \tD(\sigma)$}, \\
2^{\dim V-D(\sigma)} & \text{if $D(\sigma)=\tD(\sigma)$.}
\end{cases}$
\item
$I_8=
\begin{cases}
0 & \text{if $D(\sigma)\leq \tD(\sigma)$}, \\
2^{\dim V-1-\tD(\sigma)} & \text{if $D(\sigma)>\tD(\sigma)$.}
\end{cases}$
\end{enumerate}
\end{prop}

\begin{remark}
Parts (d)--(g) of the above proposition show that the $DD$-invariant
is recoverable from the collection of $I_j$-invariants.  Parts
(h)--(l), together with (a), show that the $I_j$-invariants are all
recoverable from the $DD$-invariant.  We have chosen to build the
paper around the $DD$-invariant---as opposed to some other collection
from the list---only because the $DD$-invariant seemed to be the most
accessible.  As it simply amounts to computing the ranks of four
matrices, it is somewhat easier to handle than
the other invariants in the list (though to be frank, all of the
invariants can be computed by linear algebra and so none are
particular difficult).

The relationship between the $I_j$ invariants and the $DD$-invariant
can be summarized as follows, where the arrows indicate that one set
of invariants can be derived from another:
\[
\xymatrixrowsep{1pc}\xymatrixcolsep{3.2pc}\xymatrix{
(I_5,I_7,I_8) \ar[r]^-{\text{(j)-(l)}} & (D,\tD) \ar[l] &
I_4 \ar[r]^-{(i)} & (\alpha,\talpha) \ar[l] \ar[r]^-{(i)} & I_3 \\
&I_1 \ar[r]^-{(h)} & \max\{D,\tD\}.\ar[l]
}
\]
The labels on the arrows refer to the relevant parts of
Proposition~\ref{pr:I-invs}.  Perhaps the only thing that requires
further explanation is the arrow $(I_5,I_7,I_8)\ra (D,\tD)$.  If we
know $I_5$, $I_7$, and $I_8$ then we know how $D$ and $\tD$ compare in
size, and we know the smaller value.  If $D=\tD$ then we therefore
know both, and if one is larger than the other then
Proposition~\ref{pr:DD-props}
 says it is larger by
exactly $1$---so again we know both.  

Observe that knowing $(I_4,I_5,I_7,I_8)$ is equivalent to knowing the
$DD$-invariant.  
\end{remark}

\begin{proof}[Proof of Proposition~\ref{pr:I-invs}]
For convenience we will write $v\cdot w$ for $b(v,w)$ in this proof.
Part (a) is trivial: the disjoint union of $S_1$, $S_2$ and $S_3$ is
the hyperplane defined by $v\cdot v=0$, and $S_4\cup S_5\cup S_6$ is
the affine hyperplane $v\cdot v=1$.

For (b), note that the subspaces defined by $v=\sigma(v)$ and
$v=\sigma(v)+\Omega$ are parallel, and likewise for the subspaces
defined by $v\cdot v=0$ and $v\cdot v=1$.  Linear algebra implies that
if $S_5$ is nonempty then it is a translate of $S_1$, and so in
particular has the same number of elements.  Likewise for $S_7$ and
$S_8$.

For (c) suppose that $I_5>0$ and $I_7>0$.  
Then there exist  vectors $v$
and $w$ such that $v\cdot v=1$, $v=\sigma v$, $\sigma
w=w+\Omega$, and $w\cdot w=0$.  Now compute that
\[ v\cdot w=\sigma v\cdot w=v\cdot \sigma w=v\cdot w+v\cdot \Omega=v\cdot
w+v\cdot v=v\cdot w+1,
\]
which is a contradiction.  The proofs for the pairs $(I_5,I_8)$ and
$(I_7,I_8)$ are entirely similar.

Parts (d) and (e) are trivial, just using the definitions of $D$ and
$\tD$.  

For (f) and (g) we prove the first biconditionals, and return to the
second biconditionals after (i).  
For (f) it is easy to prove that
$\alpha(\sigma)=0$ if and only if $I_1+I_2+I_4=2^{\dim V}$ just using
the definition of $\alpha$.  Then
use (a) to rewrite the latter condition as $I_4-I_3=2^{\dim V-1}$.  But $S_4$ is
contained in the hyperplane $v\cdot v=1$, and so certainly $I_4\leq
2^{\dim V-1}$.  Equality then follows, together with $I_3=0$. 
Similarly, for (g) one easily proves 
$I_1+I_2+I_5+I_6=2^{\dim V}$, but (a) simplifies this to $I_3+I_4=0$.

For (h), note that by (d) and (e) we have
\[ \max\{ D(\sigma),\tD(\sigma)\} = \dim
V-\log_2(I_1+\min\{I_5,I_8\}).
\]
But (c) implies that 
$\min\{I_5,I_8\}=0$.  

Now consider (i).  If $\alpha(\sigma)\neq \talpha(\sigma)$ then at
least one is zero, so by (f) and (g) $I_3=0$.  If $\alpha>\talpha$
then $\talpha=0$ and $\alpha=1$, so $I_4=0$ by (g).  If $\alpha
<\talpha$ then $\alpha=0$ and $\talpha=1$, so $I_4=2^{\dim V-1}$ by
(f).  It remains to analyze what happens when $\alpha=\talpha$.  By
Proposition~\ref{pr:DD-props} this can only happen when they both equal $1$. 
Let $M_0$ and $M_1$ be the affine subspaces of $V$
defined by $x\cdot x=0$ and $x\cdot x=1$, respectively.  Likewise, let
$N_0$ and $N_1$ be the affine subspaces defined by $x\cdot \sigma x=0$
and $x\cdot \sigma x=1$, respectively.  
Linear algebra immediately
implies that if $M_0\cap N_1$ is nonempty then it is a translate of
$M_0\cap N_0$, and likewise for $M_1\cap N_0$.  Note that by
definition $|M_0\cap N_0|=I_1+I_2$, $|M_0\cap N_1|=I_3$, and $|M_1\cap
N_0|=I_4$.  This proves that $I_3,I_4\in\{0,I_1+I_2\}$.  

We will use the fact that $V=N_0\cup N_1=M_0\cup M_1$, and that $N_0$
and $N_1$ are hyperplanes.
The assumption that
$\alpha(\sigma)=1$ says that $M_0$ (and therefore $M_1$) are
also hyperplanes in $V$.  
If $M_1\cap N_0\neq \emptyset$ then $M_0\neq N_0$, and so $M_0\cap
N_1\neq \emptyset$.  Similarly, if $M_0\cap N_1\neq \emptyset$ then
$M_1\cap N_0\neq \emptyset$.  So $I_3\neq 0$ if and only if $I_4\neq
0$.  But we know by (g) that either $I_3\neq 0$ or $I_4\neq 0$, so
they are both nonzero.  Therefore both are equal to $I_1+I_2$.
Finally, since $I_3=I_1+I_2$ it follows from (a) that $I_3=2^{\dim
  V-2}$.  

Observe that (i) immediately yields the second biconditionals in (f)
and (g).

For (j)--(l) we argue as follows.  If $D(\sigma)>\tD(\sigma)$ then by
(d) and (e) $I_5<I_8$.  So $I_8\neq 0$, which implies $I_5=I_7=0$ by
(c).  Also, (b) implies $I_8=I_1$ and so (e) gives $I_8$ in terms of
$\tD(\sigma)$.  

The argument is similar when $D(\sigma)<\tD(\sigma)$.  Finally, in the
case $D(\sigma)=\tD(\sigma)$ we know from (d) and (e) that $I_5=I_8$.
Comparing (d) and (e) to (h), we find $I_5=I_8=0$.  
Since $D(\sigma)=\tD(\sigma)$, it
follows that the subspace $T=\{v\,|\, v+\sigma(v)=\Omega\}$ has the same
dimension as $\{v\,|\, v+\sigma v=0\}$.  But the latter space  is
always nonzero, since any involution has an eigenvector with eigenvalue $1$.  
So $|T|>0$.  But $T$ is the disjoint union of $S_7$ and $S_8$, and we
know $S_8=\emptyset$.  So $I_7=|S_7|=|T|>0$.  By (b) we then have
$I_7=I_1$, and (h) then shows $I_7=2^{\dim V-D(\sigma)}$.  
\end{proof}

\section{Analysis of conjugacy classes}
\label{se:main}

In this section we will prove the main theorem of the paper, giving a
complete description of the conjugacy classes of involutions in
$\TO(2n)$.  The proof proceeds by analyzing involutions in the
semi-direct product $(\Z/2)^{2n-1}\rtimes \Sp(2n-2)$, and obtaining a
count of conjugacy classes here.  Then we produce enough matrices in
$\TO(2n)$ having different $DD$-invariants to know that these
represent all conjugacy classes.

\medskip

\subsection{Involutions in  the semi-direct product}
 Throughout this section we let  $V=nH$ and
$G_V=M_V\rtimes \Sp(V)$.  
 We will denote elements of this
group by $(x,A)$ where $x\in M_V$ and
$A\in \Sp(V)$.  We will write the map $\Sp(V)\ra \End(M_V)$ as $A\mapsto
\tilde{A}$; see (\ref{eq:action}) for the definition of this action.

\begin{prop}\label{pr:G-conj}\mbox{}\par
\begin{enumerate}[(a)]
\item The element $(x,A)$ is an involution in $M_V\rtimes \Sp(V)$
if and only if $A$ is an involution in $\Sp(V)$ and
$(\tilde{A}+\Id)(x)=0$.  
\item If $(x,A)$ is conjugate to $(y,B)$ then $A$ is conjugate to $B$
in $\Sp(V)$. 
\item Suppose $(x,A)$ is an involution and $A$ is conjugate to $B$ in
$\Sp(V)$.  Then there is a $y\in M_V$ such that $(y,B)$ is an
involution and $(x,A)$ is conjugate to $(y,B)$.  
\item $(x,A)$ is conjugate to $(y,A)$ if and only if there exists
$P\in \Sp(V)$ such that $PA=AP$ and $x+\tilde{P}y$ belongs to the
image of $\tilde{A}+\Id$.  
\end{enumerate}
\end{prop}  

\begin{proof}
Part (a) is just the calculation
\[ (x,A)\cdot
(x,A)=(x+\tilde{A}.x,A^2).
\]
Parts (b) through (d) are  similarly straightforward, and left to the reader.  
\end{proof}

The above proposition has the following significance for us.  For each
involution $A\in \Sp(V)$, let $\cS_A$ be the set of conjugacy classes
in $G_V$ that are represented by involutions of the form $(x,A)$.  If
$A$ and $B$ are conjugate involutions in $\Sp(V)$, then
$\cS_A=\cS_B$ by Proposition~\ref{pr:G-conj}(c).  
Moreover, if
$S\subseteq \Sp(V)$
is a set of representatives for the conjugacy classes in $\Sp(V)$ (one
element for each class) then the other parts of
Proposition~\ref{pr:G-conj} imply that we have
bijections
\[ (\text{conjugacy  classes in $G_V$})  \longleftrightarrow
\coprod_{A\in S} \cS_A
\]
and
\begin{myequation}
\label{eq:S_A}
 \cS_A\longleftrightarrow
\Bigl [\ker(\tilde{A}+\Id)/\im(\tilde{A}+\Id) \Bigr
]_{C(A)}
\end{myequation}
where $C(A)$ is the centralizer of $A$ in $\Sp(V)$ and we are
writing $X_{C(A)}$ for the set of orbits of $X$ under $C(A)$.  
Let us unravel the complicated-looking object on the right.  For
$x=(v,\lambda)$ in $M_V$ we have
\begin{align*}
 (\tilde{A}+\Id)(v,\lambda)=(Av,(\S
A)(v)+\lambda)+(v,\lambda)&=(v+Av,(\S A)(v))\\
& = (v+Av,q(v)+q(Av)).  
\end{align*}
This expression equals $(0,0)$ if and only if $v+Av=0$.  
So let us write $\Eig(A)=\{v\in V\,|\, Av+v=0\}$, and let $Z(A)=\Eig(A)\oplus
\F_2$.  
The group $C(A)$ acts on $Z(A)$: if $P\in C(A)$ and $(v,\lambda)\in
Z(A)$ then
\[ P.(v,\lambda)=(Pv,(\S P)(v)+\lambda).
\]
Next, let $B(A)=\{(v+Av,(\S A)(v))\,|\, v\in V\}$ and note that
$B(A)\subseteq Z(A)$.  The action of $C(A)$
on $Z(A)$ preserves $B(A)$: this comes down to the computation that if
$P\in C(A)$ then
\[ \S P + (\S A)\circ P=\S(AP)=\S(PA)=\S A + (\S P)\circ A
\]
by using Proposition~\ref{pr:Sq} twice. 

Let $H(A)=Z(A)/B(A)$, and note that the action of $C(A)$ on $Z(A)$
descends to an action on $H(A)$.  We can restate (\ref{eq:S_A}) as a bijection
\[ \cS_A \longleftrightarrow H(A)_{C(A)}.\]

To proceed further in our analysis, we will make some assumptions on
the involution $A$.  These assumptions at first might seem very
restrictive, but in fact they turn out to cover all cases.  In
particular, the assumptions in part (d) below are awkward---and almost
certainly unnecesssary.  But since they are readily seen to hold in the
cases of interest, it is easier just to make these awkward
assumptions than to somehow try to avoid them.  

\begin{lemma}  
\label{le:mainlem}
Suppose that $V$ decomposes as $V=U \oplus W$ where $W^{\perp}=U$,
and that $A\in \Sp(V)$ is of the form $K\oplus \Id_{W}$ where
$K\colon U\ra U$ is an involution such that
$\im(K+\Id_U)=\ker(K+\Id_U)$.  
Let $\pi_U\colon V\ra U$ and $\pi_W\colon V\ra W$ be the evident
projections.  
Then:
\begin{enumerate}[(a)]
\item There is a bijection $\cS_A\longleftrightarrow (W\oplus
\F_2)_{C(A)}$ where the action of $P\in C(A)$ on $(w,\lambda)$ is
given by
\[
 P.(w,\lambda)=(\pi_W(Pw),\lambda+(\S
P)(w)+(\S A)(u))
\]
where $u$ is any element of $U$ such that
$Au+u=\pi_U(Pw)$.    
\item If $W=0$ there are exactly two elements in $\cS_A$.  
\item If $W\neq 0$ and $\alpha(A)=0$, there are exactly four elements
in $\cS_A$.  
\item Suppose $W\neq 0$ and $\alpha(A)=1$.  Assume further that there
exist $u_1,u_2\in U$ such that $Ku_1=u_2$ and $b(u_1,u_2)=1$.
Then there are exactly three elements
in $\cS_A$.  
\end{enumerate}
\end{lemma}

\begin{proof}
For (a), the assumptions force $Z(A)=\im(K+\Id_U)\oplus W\oplus
\F_2$ and $B(A)=\{\bigl ((Au+u)\oplus 0,(\S A)(u)\bigr )\,|\,u\in U\}$.  So the quotient
$Z(A)/B(A)$ is clearly isomorphic to $W\oplus \F_2$.  To transplant
the action of $C(A)$ from $Z(A)/B(A)$ to  $W\oplus \F_2$,
let $(w,\lambda)\in W\oplus \F_2$ and consider the formula
\[ P.(0\oplus w,\lambda)=(Pw,(\S P)(w)+\lambda)=\bigl (\pi_U(Pw)\oplus
\pi_W(Pw),(\S P)(w)+\lambda\bigr ).
\]
Since $P\in C(A)$ it is easy to see that
$(A+I)(\pi_U(Pw))=(A+I)(Pw)=0$, 
therefore we can write $\pi_U(Pw)=Au+u$ for some $u\in U$. 
Then $(Au+u\oplus 0,(\S
A)(u))$ is in $B(A)$, and we get
\[ P.(0\oplus w,\lambda)=(\pi_W(Pw),(\S A)(u)+(\S P)(w)+\lambda)
\]
in $Z(A)/B(A)$.  This finishes the proof of (a).

Note that $(0,0)$ and $(0,1)$ in $W\oplus \F_2$ are fixed points for
the action of $C(A)$.  So $|\cS_A|\geq 2$, and one has equality if and
only if $W=0$.  This proves (b).

Let $E_0=\{(w,q(w))\,|\,w\in W-\{0\}\}$ and $E_1=
\{(w,q(w)+1)\,|\,w\in W-\{0\}\}$.  Note that
\[ W\oplus \F_2=\{(0,0)\}\amalg \{(0,1)\}\amalg E_0\amalg
E_1.\]
Assume that $W\neq 0$ and
let $w_1,w_2 \in W$ be any two nonzero elements.  By
Proposition~\ref{pr:sp-trans}
 there
exists a $P\in \Sp(W)$ such that $P(w_1)=w_2$.  
Then  $Q=\Id_U\oplus P$ is an element of $C(A)$, and
\begin{align*} 
Q.(w_1,q(w_1))&=\bigl (Qw_1,(\S Q)(w_1)+q(w_1)+0\bigr )\\
&=(Qw_1,q(w_1)+q(Qw_1)+q(w_1))\\
&=(w_2,q(w_2)).
\end{align*}
The ``$0$'' in the first line appears because $\pi_U(Qw_1)=0$, and so
we may take $u=0$ in the formula for the action given in (a).
Since $Q$ is linear and preserves $(0,1)$ we therefore also get
 $Q.(w_1,q(w_1)+1)=(w_2,q(w_2)+1)$.  These computations show that
the
elements of $E_0$ are all in the same orbit under $C(A)$, and the
elements of $E_1$ also lie in a common orbit.  Therefore, 
$(W\oplus \F_2)_{C(A)}$ has at most four
elements.  Since $(0,0)$ and $(0,1)$ are fixed points, the only
remaining question is whether points from $E_0$ and $E_1$ can ever be
in the same orbit.  

For (c), the important point is that if $\alpha(A)=0$ then 
for all $v\in V$ one has
\begin{myequation}
\label{eq:qA}
 q(Av+v)=q(Av)+q(v)+b(Av,v)=q(Av)+q(v)= (\S A)(v).
\end{myequation}
Consider the set map $h\colon W\oplus \F_2\ra \F_2$ given by
$h(w,\lambda)= q(w)+\lambda$.  Then $h$ is constant on orbits of
$C(A)$: for if $P\in C(A)$ then choose $u\in U$ such that $\pi_U(Pw)=Au+u$
and  calculate
\begin{align*}
 h\bigl (P.(w,\lambda)\bigr )& =h\bigl (\pi_W(Pw),\lambda+(\S P)(w) + (\S
 A)(u)\bigr ) \\
&=q(\pi_W(Pw))+\lambda+q(w)+q(Pw) + q(Au+u)\ \ \ \  (\text{using (\ref{eq:qA})}) \\
&=q\bigl (\pi_W(Pw)\bigr )+\lambda+q(w)+q(Pw) + q\bigl (\pi_U(Pw)\bigr
) \\
&=q(Pw)+\lambda+q(w)+q(Pw)  \\
&=\lambda+q(w). 
\end{align*}
In the second-to-last equality we have used that $q(x+y)=q(x)+q(y)$
when $b(x,y)=0$.
Notice that $h$ maps $E_0$ to $0$ and $E_1$ to $1$.  So points in
$E_0$ and $E_1$ cannot belong to the same orbit, which implies that
$|\cS_A|=4$.

Finally, assume the hypotheses for (d).  Extend $\{u_1,u_2\}$ to a
symplectic basis $\{u_1,u_2,\ldots,u_{2r-1},u_{2r}\}$ of $U$, and
choose a symplectic basis $\{w_1,\ldots,w_{2s}\}$ of $W$.  
Note that $Ku_i \in \langle
u_1,u_2\rangle^{\perp}$ for all $i\geq 3$; this is a consequence of
\[ b(Ku_i,u_1)=b(Ku_i,Ku_2)=b(u_i,u_2)=0 \]
and the parallel equation with the indices $1$ and $2$ switched.
So when $i\geq 3$ we have $Ku_i\in \langle
u_3,u_4,\ldots,u_{2r}\rangle$.  

Define $P\colon V\ra V$ as follows:

\[
\xymatrixrowsep{0.8pc}
\xymatrix{
u_1\mapsto u_1+w_2  & u_2\mapsto u_2+w_2 & u_i\mapsto u_i \
(i\geq 3) \\
w_1\mapsto u_1+u_2+w_1 & w_i\mapsto w_i \ (i\geq 2).
}
\]
It is routine to check that $P$ is an isometry and that it commutes
with $A$ (for the latter, use that $Ku_i\in 
\langle u_3,u_4,\ldots,u_{2r}\rangle$ when
$i\geq 3$).  Note that $\pi_U(Pw_1)=u_1+u_2=A(u_1)+u_1$, and
\begin{align*} 
(\S
P)(w_1)=q(w_1)+q(Pw_1)& =q(w_1)+q(w_1+(u_1+u_2))\\
&=q(w_1)+q(w_1)+q(u_1)+q(u_2)+b(u_1,u_2)\\
&= q(u_1)+q(u_2)+1.
\end{align*}
Now we use the action formula from (a) to compute:
\begin{align*}
P.(w_1,q(w_1)) &= \Bigl (w_1,q(w_1)+(\S P)(w_1) + (\S A)(u_1)\Bigr )
\\
&= \Bigl ( w_1, q(w_1)+q(u_1)+q(u_2)+1+q(u_1)+q(u_2)\Bigr ) \\
&= (w_1,q(w_1)+1).
\end{align*}
This exhibits that points from $E_0$ and $E_1$ are in the same orbit
in $(W\oplus \F_2)_{C(A)}$, therefore we have exactly three orbits.
\end{proof}

\begin{prop}
Let $A\in \Sp(V)$ be an involution.
\begin{enumerate}[(a)]
\item When $A=\Id$, $|\cS_A|=4$.
\item When $D(A)=\frac{\dim(V)}{2}$, $|\cS_A|=2$.  
\item When $\alpha(A)=1$ and $0<D(A)<\frac{\dim(V)}{2}$, $|\cS_A|=3$.  
\item When $\alpha(A)=0$ and $0<D(A)<\frac{\dim(V)}{2}$, $|\cS_A|= 4$.  
\end{enumerate}
\end{prop}

\begin{proof}
We have already proven that $|\cS_A|$ depends only on the conjugacy
class of $A$ in $\Sp(V)$.  So it suffices to prove the theorem when
$V=nH$ and $A$ ranges over the particular representatives listed in
Theorem~\ref{th:main-sp}(d).  For each of these matrices it is
transparent that there is a $(U,W,K)$ decomposition satisfying the hypotheses of
Lemma~\ref{le:mainlem}.  Moreover, for the matrices with $\alpha(A)=1$
it is transparent that the hypotheses of Lemma~\ref{le:mainlem}(d)
hold.  So the results follow immediately from Lemma~\ref{le:mainlem}.
\end{proof}

\begin{cor}
\label{co:sd-count}
If $\dim V=2n$ then $G_V$ has $5n+1$ conjugacy classes of involutions.
\end{cor}

\begin{proof} 
The proof is best explained by first looking at examples.  For $n=5$
and $n=6$ the conjugacy classes of involutions in $\Sp(2n)$ are
indicated by the dots in the following two tables:

\vspace{0.1in}

\begin{tabular}{c|cccccc}
& 0 & 1 & 2 & 3 & 4 & 5\\
\hline
$\alpha=1$ &  & $\bullet$ & $\bullet$ & $\bullet$ & $\bullet$
& $\bullet$ \\
$\alpha=0$ & $\bullet$ && $\bullet$ && $\bullet$
\end{tabular}
\qquad\qquad
\begin{tabular}{c|ccccccc}
& 0 & 1 & 2 & 3 & 4 & 5 & 6 \\
\hline
$\alpha=1$ &  & $\bullet$ & $\bullet$ & $\bullet$ & $\bullet$
& $\bullet$ & $\bullet$  \\
$\alpha=0$ & $\bullet$ && $\bullet$ && $\bullet$ && $\bullet$
\end{tabular}

\vspace{0.1in}
\noindent
(one dot for each conjugacy class).  For each involution $A\in
\Sp(V)$, mark the dot for the
conjugacy class represented by $A$ with $|\cS_A|$; this leads to the
tables

\vspace{0.1in}

\begin{tabular}{c|cccccc}
& 0 & 1 & 2 & 3 & 4 & 5\\
\hline
$\alpha=1$ &  & $\bullet_3$ & $\bullet_3$ & $\bullet_3$ & $\bullet_3$
& $\bullet_2$ \\
$\alpha=0$ & $\bullet_4$ && $\bullet_4$ && $\bullet_4$
\end{tabular}
\qquad
\begin{tabular}{c|ccccccc}
& 0 & 1 & 2 & 3 & 4 & 5 & 6 \\
\hline
$\alpha=1$ &  & $\bullet_3$ & $\bullet_3$ & $\bullet_3$ & $\bullet_3$
& $\bullet_3$ & $\bullet_2$  \\
$\alpha=0$ &$\bullet_4$ && $\bullet_4$ && $\bullet_4$ && $\bullet_2$
\end{tabular}

\vspace{0.1in}
\noindent
Adding up the numbers, there are 26 conjugacy classes of involutions
in $G_V$ when $n=5$, and $31$ conjugacy classes of involutions when
$n=6$.  

The general situation is that the dots in the first row all get
labelled with $3$, except for  column  $n$.
Likewise, the dots in the second row all get labelled with $4$, except
for column $n$.  The dots in column $n$ all get labelled with $2$.
The total of all the labels is therefore
\[ \begin{cases}
3(n-1)+2+4(\tfrac{n+1}{2}) & \text{when $n$ is odd,}\\
3(n-1)+2+4(\tfrac{n}{2}) +2 & \text{when $n$ is even.}
\end{cases}
\]
In both cases the given sum simplifies to  $5n+1$.
\end{proof}


\subsection{Involutions in \mdfn{$\TO(2n)$}}

\begin{prop}
$\TO(2n)$ has $5n-4$ conjugacy classes of involutions.  
\end{prop}

\begin{proof}
Recall that $\TO(2n)\iso (\Z/2)^{2n-1}\rtimes \Sp(2n-2)$.  By
Corollary~\ref{co:sd-count}, the number of involutions in the
semi-direct product is $5(n-1)+1$.
\end{proof}

Our next goal is to produce a collection of specific involutions 
in $\TO(2n)$ and
show that they must represent the $5n-4$ conjugacy classes.  
For the following proposition recall that $I=
\begin{bsmallmatrix} 1 &
0 \\ 0 & 1\end{bsmallmatrix}$ and $J=
\begin{bsmallmatrix} 0 &
1 \\ 1 & 0\end{bsmallmatrix}$.

\begin{prop} 
\label{pr:lower}
In $\TO(2n)$ we have the following calculations:
\begin{enumerate}[(a)]
\item $DD\bigl (I^{\oplus (n-k)}\oplus J^{\oplus k}\bigr )=[k,1,k+1,1]$ for
$1\leq k\leq n-1$.  
\item $DD\bigl (m(I^{\oplus(n-k)} )\oplus J^{\oplus k}\bigr )=
\begin{cases} 
[k+1,0,k,1] & \text{if $k$ is even,}\\
[k+1,0,k+1,1] & \text{if $k$ is odd}
\end{cases}$
\par\noindent
for $0\leq k\leq n-1$.  
\item $DD\bigl ( m(I^{\oplus(n-k-1)} \oplus J^{\oplus k}) \oplus J
\bigr )=[k+2,1,k+2,1]$ for $1\leq k\leq n-2$.  
\end{enumerate}
\end{prop}

\begin{proof}
These are all simple computations.  We only do (b), since the others
are similar (and easier).  Let $A=m(I^{\oplus(n-k)} )\oplus J^{\oplus
  k}$.  Since $A$ only has zeros along its diagonal, $\alpha(A)=0$.
Since $m(A)$ has a $1$ (and in fact, all ones) along its diagonal,
$\talpha(A)=1$.  
The matrices $A+\Id$ and $m(A)+\Id$ have the form
\[ 
\begingroup
\renewcommand*{\arraystretch}{0.8}
\begin{bmatrix}
1 & 1 & \cdots & 1 & \\
\vdots & \vdots && \vdots \\
1 & 1 & \cdots & 1 \\
 &  &  &  & 1 & 1 &  \\
 &  & & & 1 & 1 &  \\
 &  &  &  &  &  & \ddots &  &  &  \\
 &  &  &  &  &  &  & \ddots &  &  \\
 &  &  &  &  &  & && 1 & 1  \\
 &  &  &  &  &  & && 1 & 1  \\
\end{bmatrix}\ \ 
\begin{bmatrix}
&&&& 1 & 1 & \cdots & 1 & 1 & 1  \\
&&&& \vdots & \vdots & &\vdots &\vdots & \vdots \\
&&&& 1 & 1 & \cdots & 1 & 1 & 1  \\
1 & 1 & \cdots & 1 & 0 & 0 & 1 &1 & \cdots & 1\\
1 & 1 & \cdots & 1 & 0 & 0 & 1 & 1& \cdots & 1\\
1 & 1 & \cdots & 1 & 1 & 1 & 0 & 0 & \cdots & 1\\
1 & 1 & \cdots & 1 & 1 & 1 & 0 & 0 & \cdots & 1\\
\vdots & \vdots && \vdots &&&& \ddots & \ddots \\
1 & 1 & \cdots & 1 & 1 & 1 & 1& \cdots & 0 & 0\\
1 & 1 & \cdots & 1 & 1 & 1 & 1& \cdots & 0 & 0\\
\end{bmatrix}
\endgroup
\]
The former clearly has rank $k+1$.  For the latter, row reduce the
matrix by adding row 1 to the bottom $2k$ rows.  This gives a new
matrix where the lower $2k$ rows clearly have rank $k$.  
The question then becomes whether row one
of the matrix is a linear combination of these new lower $2k$ rows.
It is clear that this is the case precisely when $k$ is even.

As an alternative to just doing the rank computations, one can use
Theorem~\ref{th:DD-osum} from the next section (but this is not
really easier).  
\end{proof}

\begin{cor}
\label{co:lower}
The matrices listed in Proposition~\ref{pr:lower}, together with the
mirrors of the matrices in (a) and (b),
represent all the conjugacy classes of involutions in $\TO(2n)$.
\end{cor}

\begin{proof}
By Proposition~\ref{pr:mirror-conjugacy}, 
the $DD$-invariants are constant on conjugacy classes.
Moreover, if $DD(A)=[a,b,c,d]$ then $DD(mA)=[c,d,a,b]$, simply by the
definition.  
A look at the $DD$-invariants that appear in
Proposition~\ref{pr:lower} reveals that there are no overlaps between
parts (a), (b), and (c), even  if one includes the mirrors of the
matrices in (a) and (b).  Now we count.  There are $n-1$ matrices
covered by (a), which becomes $2n-2$ when one includes their mirrors.
There are $n$ matrices covered by (b), becoming $2n$ when one includes
mirrors.
Finally, there are $n-2$ matrices covered by (c).  So the total number
of matrices is
\[ 2n-2+2n+n-2=5n-4,
\]
and these represent distinct conjugacy classes.
\end{proof}

\begin{cor}
Two involutions in $\TO(2n)$ are conjugate if and only if they have
the same $DD$-invariant.
\end{cor}

\begin{proof}
Immediate from Proposition~\ref{pr:lower} and Corollary~\ref{co:lower}.
\end{proof}

The results in this section together constitute a proof of
Theorem~\ref{th:main-TO} from the introduction.


\section{The $DD$-invariant and direct sums}
\label{se:dsum}

Suppose that $(U,b_U)$ and $(W,b_W)$ are two bilinear spaces over
$\F_2$, and $\sigma\in \Iso(U)$ and $\theta\in \Iso(W)$ are two
involutions.  It is natural to ask how the conjugacy class of the involution
$\sigma\oplus \theta\colon U\oplus W\ra U\oplus W$ depends on the
conjugacy classes of $\sigma$ and $\theta$.  Answering this is
important for concrete computations, and it is needed for the 
applications in \cite{D}.

Unfortunately, stating the answer to the question is a little awkward
due to the variety of cases that can occur.  From the point of view of
classifying involutions there are three types of bilinear spaces:
symplectic, even-dimensional orthogonal, and odd-dimensional
orthogonal.  This leads to six different cases that must be analyzed
for the pair $(U,W)$.  And as the classification of conjugacy classes
of involutions looks slightly different for the three types,
the bookkeeping to handle the direct sum is somewhat clunky.  

In this section we try, to the extent possible, to unify the three cases
into a common classification system.   The end result is still a bit
clunky, but it is manageable.  

\medskip

\subsection{Unification}
For brevity let us write SYMP, EVO, and ODDO for the three types
of bilinear spaces over $\F_2$.  Note that in this nomenclature direct
sums behave as in the chart below:

\vspace{0.1in}

\begingroup
\renewcommand*{\arraystretch}{1.2}

\qquad\qquad
\qquad\qquad
\begin{tabular}{|c||c|c|c|}
\hline
$\oplus$ & $\SYMP$ & $\ODDO$ & $\EVO$ \\ 
\hhline{|=||=|=|=|}
$\SYMP$ & $\SYMP$ & $\ODDO$ & $\EVO$ \\ \hline
$\ODDO$ & $\ODDO$ & $\EVO$ & $\ODDO$ \\ \hline
$\EVO$ & $\EVO$ & $\ODDO$ & $\EVO$ \\ \hline
\end{tabular}
\endgroup

\vspace{0.15in}

For all three types we have a distinguished vector $\Omega$ in the
bilinear space, uniquely characterized by the property that
$\Omega\cdot v=v\cdot v$, for all vectors $v$.  When the bilinear
space is symplectic one has $\Omega=0$.  If $U$ and $W$ are bilinear
spaces and $V=U\oplus W$, one readily checks that
$\Omega_V=\Omega_U+\Omega_W$.

We extend the $DD$-invariant to the SYMP and ODDO cases in a trivial
way that we will now explain.  
Let $(V,b_V)$ be a bilnear space and $\sigma\in \Iso(V)$ be an
involution.  If $V$ is symplectic then define $\tD(\sigma)=D(\sigma)$,
$\talpha(\sigma)=\alpha(\sigma)$, and
\[ DD(\sigma)=[D(\sigma),\alpha(\sigma),D(\sigma),\alpha(\sigma)].
\]
If $V$ is ODDO then $\sigma$ always preserves $\Omega$ and
$b(\sigma(\Omega),\Omega)=b(\Omega,\Omega)=1$.  So the usual
definition of the $\alpha$-invariant is not useful here.  
To get a more useful invariant,
note that $\langle \Omega\rangle^{\perp}\subseteq V$ is
symplectic and $\sigma$ restricts to a map 
$\sigma'\colon\langle\Omega\rangle^{\perp} \ra 
\langle\Omega\rangle^{\perp}$.  Define
$\alpha(\sigma)=\alpha(\sigma')$.  Since $\sigma(\Omega)=\Omega$ it
follows at once that 
$D(\sigma)=D(\sigma')$, so the change to $\sigma'$ is really just for
the purposes of the $\alpha$-invariant.  Define
$\tD(\sigma)=D(\sigma)$, $\talpha(\sigma)=\alpha(\sigma)$, and 
\[ DD(\sigma)=[D(\sigma),\alpha(\sigma),D(\sigma),\alpha(\sigma)]=
[D(\sigma'),\alpha(\sigma'),D(\sigma'),\alpha(\sigma')].
\]

We can now say by Theorem~\ref{th:main-sp}, Theorem~\ref{th:main-TO}, 
and Proposition~\ref{pr:TO-odd} that the
$DD$-invariant completely separates the conjugacy classes of orbits in
each of the SYMP, ODDO, and EVO cases.  Of course, in the first two
cases the $DD$-invariant contains very redundant information.  

The following lemma will be needed in the next section:

\begin{lemma}
\label{le:ODDO-alpha}
Assume $(W,b)$ is ODDO, and that $\sigma$ is an involution in
$\Iso(W)$.  Then the following three statements are equivalent:
\begin{enumerate}[(1)]
\item $\alpha(\sigma)=1$.
\item There exists $w\in W$ such that $b(w,w)=0$ and $b(w,\sigma
w)=1$.
\item There exists $v\in W$ such that $b(v,v)=1$ and $b(v,\sigma
v)=0$.  
\end{enumerate}
\end{lemma}

\begin{proof}
The equivalence of (1) and (2) is just the definition of
$\alpha(\sigma)$.  If $b(w,w)=0$ and $b(w,\sigma w)=1$ then let
$v=w+\Omega$.  Then $\sigma v=\sigma w+\Omega$.  One readily checks
that $b(v,v)=b(\Omega,\Omega)=1$ and $b(v,\sigma v)=b(w,\sigma
w)+b(\Omega,\Omega)=0$.    So (2) implies (3), and the converse is
similar.  
\end{proof}

\subsection{Direct sums}
It is trivial to check that in all cases
$D(\sigma\oplus\theta)=D(\sigma)+D(\theta)$.  It is also trivial to
see that when $U$ and $W$ are both even-dimensional, then
$\alpha(\sigma\oplus \theta)=\max\{\alpha(\sigma),\alpha(\theta)\}$.  
Our ``baseline'' for how $DD(\sigma\oplus\theta)$ relates to
$DD(\sigma)$ and $DD(\theta)$ is that
$DD(\sigma\oplus\theta)=DD(\sigma)\# DD(\theta)$ where
for tuples $X,Y\in \Z\times \Z/2\times \Z\times \Z/2$ we define
\[ X \# Y =[X_1+Y_1,\max\{X_2,Y_2\},X_3+Y_3,\max\{X_4,Y_4\}].
\]
By ``baseline'' we simply mean that this is the result that holds in
the majority of cases, and the exceptional cases can be seen as small
deviations from this baseline.

\begin{thm}
\label{th:DD-osum}
Let $(U,b_U)$ and $(W,b_W)$ be two bilinear spaces over $\F_2$.  
Let $\sigma\in \Iso(U)$ and $\theta\in \Iso(W)$ be two involutions.
Then
\begin{enumerate}[(a)]
\item $DD(\sigma\oplus \theta)=DD(\sigma)\# DD(\theta)$ if either $U$
or $W$ is SYMP.
\item If $U$ and $W$ are both ODDO then
\begin{align*} DD(\sigma\oplus
\theta)&=[D(\sigma)+D(\theta), 1,
D(\sigma)+D(\theta)+1,\max\{\alpha(\sigma),\alpha(\theta)\}]\\
&= \bigl [ DD(\sigma)\# DD(\theta) \bigr ] \# [0,1,1,0].
\end{align*}
\item If $U$ is ODDO and $W$ is EVO then
\[ DD(\sigma\oplus
\theta)=[D(\sigma)+D(\theta),\max\{\alpha(\sigma),\talpha(\theta)\},
D(\sigma)+D(\theta),\max\{\alpha(\sigma),\talpha(\theta)\}].
\]
\item If $U$ and $W$ are both EVO then 
\[ D(\sigma\oplus \theta)=[DD(\sigma)\# DD(\theta) ] + E \]
where $+$ means componentwise-addition and 
\[ E=
\begin{cases}
[0,0,0,0] & \text{if $\tD(\sigma)=D(\sigma)$ or
  $\tD(\theta)=D(\theta)$,} \\
[0,0,-1,0]  & \text{if $\tD(\sigma)>D(\sigma)$ and
  $\tD(\theta)>D(\theta)$,} \\
[0,0,1,0] & \text{if $\tD(\sigma)>D(\sigma)$ and
  $\tD(\theta)<D(\theta)$,} \\
[0,0,1,0] & \text{if $\tD(\sigma)<D(\sigma)$ and
  $\tD(\theta)>D(\theta)$,} \\
[0,0,2,0] & \text{if $\tD(\sigma)<D(\sigma)$ and
  $\tD(\theta)<D(\theta)$.}
\end{cases}
\]

\end{enumerate}
\end{thm}

\begin{remark}
In part (d), the main point is the behavior of the $\tD$ invariant.  
Here is a bookkeeping system that contains the same information as the
five cases listed in (d).
Let $C$ be the monoid $\{-1,0,1\}$ with integer multiplication.  Every
involution $\sigma\in \TO(2k)$ may be given a ``charge''
$c(\sigma)$ in $C$ as follows.  
If $\tD(\sigma)>D(\sigma)$ then $c(\sigma)=0$.  If
$\tD(\sigma)=D(\sigma)$ then $c(\sigma)=1$.  If
$\tD(\sigma)<D(\sigma)$ then $c(\sigma)=-1$.  
Under this system one has $c(\sigma\oplus\theta)=c(\sigma)c(\theta)$,
where the multiplication of charges takes place in $C$.  This formula
suggests that
the $4$-tuple $[D(\sigma),c(\sigma),\alpha(\sigma),\talpha(\sigma)]$
might be a more convenient fundamental system of invariants for involutions, 
as opposed to the
$DD$-invariant.  We have not gone this route mainly because the
definition of $c(\sigma)$ is not particularly intuitive, and in
practice it would usually be computed via $\tD(\sigma)$ anyway.  
\end{remark}

\begin{proof}[Proof of Theorem~\ref{th:DD-osum}]
This proof is somewhat long and clunky, due to the number of cases.
As we remarked before, the $D$-invariant is always additive---so we
will ignore it for the remainder of the proof, and concentrate
on the other three invariants.  Set $V=U\oplus W$, and note that
$\Omega_V=\Omega_U+\Omega_W$.  

For part (a) we assume that $U$ is symplectic.  There are then three
cases, depending on the type of $W$.  If $W$ is also symplectic then
the result is easy.  Assume that $W$ is EVO, so that $V$ is also
EVO.    
Observe that $\tD(\sigma\oplus\theta)$ is the
dimension of the space
\[ \bigl \{(u+\sigma u+(b(u,u)+b(w,w))\Omega_U, w+\theta w +
(b(u,u)+b(w,w))\Omega_W)\,|\, u\in U, w\in W\bigr \}.
\]
But $\Omega_U=0$ and
$b(u,u)=0$ for all $u\in U$, so this simplifies to
\[
\bigl \{(u+\sigma u, w+\theta w +
b(w,w)\Omega_W)\,|\, u\in U, w\in W\bigr \}
\]
which splits as
\[
\{ u+\sigma u\,|\,u\in U\} \oplus \{w+\theta w+b(w,w)\Omega_W\,|\, w\in W\}.
\] 
The dimensions of the two summands are $D(\sigma)=\tD(\sigma)$ and
$\tD(\theta)$, respectively.  So
$\tD(\sigma\oplus\theta)=\tD(\sigma)+\tD(\theta)$.  

One has $\alpha(\sigma\oplus\theta)=1$ if and only if there exist
$u\in U$, $w\in W$ such that
\[ 1=b(u+w,\sigma u+\theta w)=b(u,\sigma u)+b(w,\theta w) \]
and clearly this has a solution if and only if either
$\alpha(\sigma)=1$ or $\alpha(\theta)=1$.  So
$\alpha(\sigma+\theta)=\max\{\alpha(\sigma),\alpha(\theta)\}$.  

Likewise,
$\talpha(\sigma\oplus \theta)=1$ if and only if there exist $u\in U$,
$w\in W$ such that
\begin{align*} 
1&=b(u+w,\sigma u+\theta w+ (b(u,u)+b(w,w))\Omega_V) \\
&= b(u,\sigma u)+b(w,\theta w)+b(u,u)+b(w,w) \\
&= b(u,\sigma u) + b(w,\theta w) + b(w,w) \ \ \text{since $U$ is symplectic}
\\
&= b(u,\sigma u)+ b(w,\theta w+b(w,w)\Omega_W).
\end{align*}
Clearly such $u$ and $w$ exist if and only if either
$\alpha(\sigma)=1$ or $\talpha(\theta)=1$.  Since
$\alpha(\sigma)=\talpha(\sigma)$, we can write
$\talpha(\sigma\oplus\theta)=\max\{\talpha(\sigma),\talpha(\theta)\}$.
This finishes the proof when $W$ is EVO.   

To complete the proof for (a), assume that $U$ is SYMP and $W$ is
ODDO.  Here $V$ is ODDO, so $\tD$ and $\talpha$ are redundant---it
only remains for us to compute $\alpha$ for $\sigma\oplus \theta$.  We have
$\alpha(\sigma\oplus \theta)=1$ if and only if there exists $u\in U$,
$w\in W$ such that $0=b(u+w,\Omega_V)=b(w,\Omega_W)$ and
\begin{align*} 
1 = b(u+w,\sigma u+\theta w  )
  = b(u,\sigma u) + b(w,\theta w).  
\end{align*}
Having a $u$ such that $b(u,\sigma u)=1$ is equivalent to
$\alpha(\sigma)=1$.  Having a $w$ such that $b(w,\Omega_W)=0$ and
$b(w,\theta w)=1$ is equivalent to 
$\alpha(\theta)=1$.  So $\alpha(\sigma\oplus
\theta)=\max\{\alpha(\sigma),\alpha(\theta)\}$.  

For (b), assume that $U$ and $W$ are ODDO.  Here
$\Omega_V=\Omega_U+\Omega_W$.  
We readily compute that
\[ b\bigl ( (\sigma\oplus \theta)(\Omega_U), \Omega_U\bigr)=b\bigl (\sigma(\Omega_U),
\Omega_U\big ) =b(\Omega_U, \Omega_U)=1,
\]
so $\alpha(\sigma\oplus \theta)=1$.  
Let $F=m(\sigma\oplus\theta)$.  Recall that $F\colon V\ra V$ is
the map given by $(\sigma\oplus\theta)(v)+b(v,v)\Omega_V$.  
But we can decompose $V$ as $V=\langle \Omega_U\rangle^{\perp}\oplus
\langle \Omega_W\rangle^{\perp} \oplus \langle \Omega_U\rangle \oplus
\langle \Omega_W\rangle$.  The first two summands are symplectic, so
on these $F$ agrees with $\sigma$ and $\theta$, respectively.  On the
last two summands $F$ is readily checked to satisfy
$F(\Omega_U)=\Omega_W$ and $F(\Omega_W)=\Omega_U$.   It follows at
once that
$\talpha(\sigma\oplus\theta)=\alpha(F)=\max\{\alpha(\sigma),\alpha(\theta)\}$.
 Moreover, $\tD(\sigma\oplus\theta)$ is the dimension of $\im(F+\Id)$,
 which clearly decomposes as $\im(\sigma+\Id)\oplus \im(\theta\oplus
 \Id) \oplus \langle\Omega_V\rangle$.  So
 $\tD(\sigma\oplus\theta)=D(\sigma)+D(\theta)+1=\tD(\sigma)+\tD(\theta)+1$.

Now we turn to (c), so assume $U$ is ODDO and $V$ is EVO.  Then
$U\oplus V$ is ODDO, so $\tD(\sigma\oplus\theta)=D(\sigma\oplus \theta)$ and
$\talpha(\sigma\oplus \theta)=\alpha(\sigma\oplus\theta)$.  We only need to
compute $\alpha(\sigma\oplus\theta)$.  This invariant is equal to $1$
if and only if there exist $u\in U$, $w\in W$ such that 
\addtocounter{subsection}{1}
\begin{align}
\label{eq:ODDO-eqs}
&0=b(u+w,u+w)= b(u,u)+b(w,w), \text{\ and}\\
& 1=b(u+w,\sigma u+\theta w)= b(u,\sigma u)+b(w,\theta w). \notag
\end{align}
These equations break down into four possibilities:

\vspace{0.1in}

\begin{tabular}{r|c|c|c|c}
& I & II & III & IV \\ \hline
$b(u,u), \ b(w,w)$ & 0, 0 & 0, 0 & 1, 1 &
1, 1\\
$b(u,\sigma u), \ b(w,\theta w)$ & 1, 0 & 0,1 & 1,0 & 0,1
\end{tabular}

\vspace{0.1in}
\noindent
In case I we have $\alpha(\sigma)=1$, by definition of $\alpha$.  In case II we have
$I_3(\theta)> 0$, and so $\talpha(\theta)=\alpha(\theta)=1$ by
Proposition~\ref{pr:I-invs}(f,g).  In case III we have
$I_4(\theta)>0$, so 
$\talpha(\theta)=1$ by Proposition~\ref{pr:I-invs}(g).  And in case IV we have
$\alpha(\sigma)=1$, by Lemma~\ref{le:ODDO-alpha}.  In all cases we
have either $\alpha(\sigma)=1$ or $\talpha(\theta)=1$.

Conversely, if $\alpha(\sigma)=1$ then we have a $u\in U$ such that
$b(u,u)=0$ and $b(u,\sigma u)=1$.  Then the pair $(u,0)$ is 
a solution to (\ref{eq:ODDO-eqs}).  Likewise, if $\talpha(\theta)=1$
then by Proposition~\ref{pr:I-invs}(g) $I_4(\theta)>0$; so 
there exists $w\in W$ such that $b(w,w)=1$ and $b(w,\sigma w)=0$.
Then $(\Omega_U,w)$ is a solution to (\ref{eq:ODDO-eqs}).  So we have
now proven that
\[ \alpha(\sigma\oplus \theta)=1 \iff (\text{$\alpha(\sigma)=1$ or
  $\talpha(\theta)=1$}).
\]
This is equivalent to $\alpha(\sigma\oplus
\theta)=\max\{\alpha(\sigma),\talpha(\theta)\}$.  This completes (c).

Finally, we turn to (d).  
The computations of $\alpha$ and $\talpha$ for
$\sigma\oplus\theta$ are straightforward and left to the reader.
It remains to deal with $\tD$.  
Let $F=m(\sigma\oplus\theta)+\Id$.  Recall that
$F\colon V\ra V$ is the map given by $v\mapsto
(\sigma\oplus\theta)(v)+v+b(v,v)\Omega_V$.  Let $M$ be the image of $F$, so that
$\tD(\sigma\oplus\theta)=\dim M$.  

Note that $\Omega_V=\Omega_U\oplus
\Omega_W$.  Define
\[ P=F(U)=\{\sigma u+u+b_U(u,u)\Omega_U + b_U(u,u)\Omega_W\,|\, u\in
U\}
\]
and
\[ Q=F(V)=\{\theta w+w+b_W(w,w)\Omega_W + b_W(w,w)\Omega_U\,|\, w\in
W\}.
\]
Then $M=P+Q$, and clearly $P\cap Q\subseteq \langle
\Omega_U,\Omega_W\rangle$.  

We claim that
\addtocounter{subsection}{1}
\begin{align}
\label{eq:eqs}
 &\Omega_U\in P \iff I_7(\sigma)\neq 0, \quad \Omega_W\in Q\iff
I_7(\theta)\neq 0\\ \notag
&\Omega_W\in P \iff I_8(\sigma)\neq 0, \quad \Omega_U\in Q\iff
I_8(\theta)\neq 0, \\ \notag
&\Omega_U+\Omega_W\in P\iff I_5(\sigma)\neq 0, \quad
\Omega_U+\Omega_W\in Q\iff I_5(\theta)\neq 0.
\end{align}
These are all easy statements.  For example, clearly $\Omega_U\in P$ if and
only if there exists a $u\in U$ such that $b_U(u,u)=0$ and
$\sigma(u)+u=\Omega_U$.  This is precisely the condition that
$I_7(\sigma)\neq 0$.  The other statements are similar.

Suppose that $I_7(\sigma)=I_7(\theta)=0$.  By
Proposition~\ref{pr:I-invs}(k) this is the assumption that $D(\sigma)\neq
\tD(\sigma)$ and $D(\theta)\neq \tD(\theta)$.  Also by
Proposition~\ref{pr:I-invs}, either $I_5(\sigma)$ or $I_8(\sigma)$ is
nonzero, and similarly for  $\theta$.  So we have ($\Omega_W\in P$ or
$\Omega_U+\Omega_W\in P$) and ($\Omega_U\in Q$ or
$\Omega_U+\Omega_W\in Q$).  Note that all four combinations lead to
$\Omega_U+\Omega_W\in P+Q$.   
If $I_8$ is nonzero for
either $\sigma$ or $\theta$ 
then one readily
checks using (\ref{eq:eqs}) that $\Omega_U,\Omega_W\in P+Q$.
Therefore
\[ M = \langle \Omega_U,\Omega_W\rangle + 
\{\sigma(u)+u\,|\,u\in U\}
+ \{\sigma(w)+w\,|\,w\in W\}.
\]
Note that the second space has dimension $D(\sigma)$, and the third
space has dimension $D(\theta)$.  Moreover,
$I_8(\sigma)\neq 0$ if and only if $\Omega_U$ is in the second space, and
$I_8(\theta)\neq 0$ if and only if  $\Omega_W$ is in the third space.
We will use these observations to analyze $\dim M$ in the various cases.  

If $\tD(\sigma)>D(\sigma)$ and $\tD(\theta)<D(\theta)$ then by
Proposition~\ref{pr:I-invs} we know $I_8(\sigma)=0$,
$I_8(\theta)\neq 0$, and $I_7(\sigma)=I_7(\theta)=0$.  So
\[ \dim M=1+D(\sigma)+D(\theta)=1+(\tD(\sigma)-1)+(\tD(\theta)+1)=\tD(\sigma)+\tD(\theta)+1.
\]
The analysis is identical in the opposite case $\tD(\sigma)<D(\sigma)$
and $\tD(\theta)>D(\theta)$.  
If $\tD(\sigma)<D(\sigma)$ and $\tD(\theta)<D(\theta)$ then
$I_7(\sigma)=I_7(\theta)=0$ and {\it both\/}
$I_8(\sigma)$ and $I_8(\theta)$ are nonzero, therefore $\dim
M=D(\sigma)+D(\theta)=\tD(\sigma)+\tD(\theta)+2$.  

Next assume 
$\tD(\sigma)>D(\sigma)$ and $\tD(\theta)>D(\theta)$.  Then
by Proposition~\ref{pr:I-invs} we know
 both
$I_5(\sigma)$ and $I_5(\theta)$ are nonzero.    
So $\Omega_U+\Omega_W\in P\cap Q$ and we can write
\[ M=\langle \Omega_U+\Omega_W\rangle + 
\{\sigma(u)+u\,|\,u\in U\}
+ \{\sigma(w)+w\,|\,w\in W\}.
\]
Since $I_7(\sigma)=I_8(\sigma)=0$, $\Omega_U$ is not contained in the
middle subspace.  Similarly, $\Omega_W$ is not contained in the right
subspace.  It follows that the above is a direct sum decomposition of $M$,
and so $\dim M=1+D(\sigma)+D(\theta)=\tD(\sigma)+\tD(\theta)-1$.  

We only have left to analyze the case where $\tD(\sigma)=D(\sigma)$
(or the parallel case where $\sigma$ and $\theta$ are interchanged).  
Here $I_7(\sigma)\neq 0$ and $I_5(\sigma)=I_8(\sigma)=0$.  
So $\Omega_U\in P$ and we can therefore write
\[ M = 
\langle \Omega_U\rangle + \{\sigma(u)+u+b_U(u,u)\Omega_W\,|\,
u\in U\} + \{\theta(w)+w+b_W(w,w)\Omega_W\,|\,w\in W\}.
\]
The dimension of the third summand is $\tD(\theta)$.  Because
$I_7(\sigma)\neq 0$, $\Omega_U$ lies in the second summand and so the
$\langle \Omega_U\rangle$ piece can be ignored.  The second summand is
contained in $\{\sigma(u)+u\,|\,u\in U\}\oplus \langle
\Omega_W\rangle$, but since $I_5(\sigma)= 0$ it does not contain
$\Omega_W$.  So its dimension is clearly the same as
$\{\sigma(u)+u\,|\,u\in U\}$, which is $D(\sigma)$.  We also get
\[ M=
\{\sigma(u)+u+b_U(u,u)\Omega_W\,|\,
u\in U\} 
\oplus \{\theta(w)+w+b_W(w,w)\Omega_W\,|\,w\in W\},
\]
since the only vector that could possibly lie in the intersection is
$\Omega_W$ and we have just observed it is not in the left summand.  
So
$\dim M = D(\sigma)+\tD(\theta)=\tD(\sigma)+\tD(\theta)$.
\end{proof}


\bibliographystyle{amsalpha}

\end{document}